\documentclass{article}[12pt]
\usepackage[sf]{titlesec}

\usepackage{amssymb}
\usepackage{latexsym}
\usepackage{amsmath, amsfonts, amsthm}
\usepackage{xspace}
\usepackage{multicol}
\usepackage{colordvi}
\usepackage{epsf}
\usepackage{graphicx}
\usepackage{color,psfrag} 
\usepackage{xcolor}
\usepackage{graphics}
\usepackage{amsthm}
\usepackage{amscd}
\usepackage{sidecap}
\usepackage{wrapfig}
\usepackage{subfig}
\usepackage{enumitem}

\newtheoremstyle{slplain}
{.1\baselineskip\@plus.2\baselineskip\@minus.2\baselineskip}
{.1\baselineskip\@plus.2\baselineskip\@minus.2\baselineskip}
{}
{}
{\bfseries}
{.}
{ }
{}

\theoremstyle{slplain}

\titlespacing{\section}{0pt}{*0.45}{*0.45}
\titlespacing{\subsection}{0pt}{*0.45}{*0.45}
\titlespacing{\subsubsection}{0pt}{*0.45}{*0.45}

\newcounter{firstbib}

\newcommand{\la}{\left \langle}
\newcommand{\ra}{\right\rangle}

\newcommand{\norm}[1]{\left\lVert #1 \right\rVert}

\newtheorem{theorem}{Theorem}[section]
\newtheorem{corollary}[theorem]{Corollary}
\newtheorem{lemma}[theorem]{Lemma}

\theoremstyle{definition}

\newtheorem{assumption}[theorem]{Assumption}

\theoremstyle{remark}
\newtheorem{remark}[theorem]{Remark}

\numberwithin{equation}{section}

\usepackage{mathtools}

\DeclarePairedDelimiter\floor{\lfloor}{\rfloor}

\title{Mean Field Analysis of  Deep Neural Networks}
\author{Justin Sirignano\footnote{Department of Industrial \& Systems Engineering, University of Illinois at Urbana-Champaign, Urbana, E-mail: jasirign@illinois.edu} \phantom{.}  and Konstantinos Spiliopoulos\footnote{Department of Mathematics and Statistics, Boston University, Boston, E-mail: kspiliop@math.bu.edu}
\thanks{K.S. was partially supported by the National Science Foundation (DMS 1550918) and Simons Foundation Award 672441}\\
}

\begin{document}

\maketitle

\begin{abstract}
We analyze multi-layer neural networks in the asymptotic regime of simultaneously (A) large network sizes and (B) large numbers of stochastic gradient descent training iterations. We rigorously establish the limiting behavior of the multi-layer neural network output. The limit procedure is valid for any number of hidden layers and it naturally also describes the limiting behavior of the training loss. The ideas that we explore are to (a) take the limits of each hidden layer sequentially and (b) characterize the evolution of parameters in terms of their initialization. The limit satisfies a system of deterministic integro-differential equations. The proof uses methods from weak convergence and stochastic analysis. We show that, under suitable assumptions on the activation functions and the behavior for large times, the limit neural network recovers a global minimum (with zero loss for the objective function).
\end{abstract}

\section{Introduction}

Machine learning, and in particular deep learning, has achieved immense success, revolutionizing fields such as image, text, and speech recognition. It is also increasingly being used in engineering, medicine, and finance. However, despite their success in practice, there is currently limited mathematical understanding of deep neural networks. This has motivated recent mathematical research on multi-layer learning models such as  \cite{Ruoyu1}, \cite{Ruoyu2}, \cite{Ruoyu3}, \cite{Lee1}, \cite{Du1}, \cite{Mallat}, \cite{NeuralNetworkLLN}, \cite{NeuralNetworkCLT}, \cite{Montanari}, and \cite{Rotskoff_VandenEijnden2018}.

Neural networks are nonlinear statistical models whose parameters are estimated from data using stochastic gradient descent (SGD) methods. Deep learning uses neural networks with many layers (i.e., ``deep" neural networks), which produces a highly flexible, powerful and effective model in practice. Typically, a neural network with multiple layers between the input and the output layer is called a ``deep" neural network, see for example \cite{Goodfellow}. We analyze multi-layer neural networks that have a fixed number of layers between the input and output layer, and where the number of hidden units in each layer becomes large.

Applications of deep learning include image recognition (see \cite{LeCun} and \cite{Goodfellow}), facial recognition \cite{FacialRecognition}, driverless cars \cite{DriverlessCar}, speech recognition (see \cite{LeCun}, \cite{DeepVoice}, \cite{GoogleDuplex}, and \cite{SpeechRecognition3}), and text recognition (see \cite{GoogleMachineTranslation} and \cite{MachineTranslation2}). Neural networks have also been applied in engineering, robotics, medicine, and finance (see \cite{Ling1}, \cite{Ling2}, \cite{Robotics1}, \cite{Robotics2}, \cite{Robotics3}, \cite{NatureMedicine2}, \cite{Finance1}, \cite{Finance2}, \cite{SirignanoSpiliopoulos2017}, and \cite{SirignanoSpiliopoulos2017_DGM}).

In this paper we characterize multi-layer neural networks in the asymptotic regime of large network sizes and large numbers of stochastic gradient descent iterations. We rigorously prove the limit of the neural network output as the number of hidden units increases to infinity. The proof relies upon weak convergence analysis for stochastic processes. The result can be considered a ``law of large numbers" for the neural network's output when both the network size and the number of stochastic gradient descent steps grow to infinity. We show that the neural network output in the large hidden-units and large SGD-iterates limit depends on paths of representative weights that go from input to output layer. This result is then used to show that, under suitable assumptions, the limit neural network seeks to minimize the limit objective function and achieve zero loss.

Recently, law of large numbers and central limit theorems have been established for neural networks with a single hidden layer \cite{Chizat2018,Szpruch2019,Montanari,Rotskoff_VandenEijnden2018,NeuralNetworkLLN,NeuralNetworkCLT}. 
For a single hidden layer, one can directly study the weak convergence of the empirical measure of the parameters. However, in a neural network with multiple layers, there is a closure problem when studying the empirical measure of the parameters (which is explained in Section \ref{ChallengesIntro}). Consequently, the law of large numbers for a multi-layer network is not a straightforward extension of the single-layer network result and the analysis involves unique challenges which require new approaches. In this paper we establish the limiting behavior of the output of the neural network.

To illustrate the idea, we consider a multi-layer neural network with two hidden layers:
\begin{eqnarray}
g_{\theta}^{N_1, N_2}(x) = \frac{1}{N_{2}} \sum_{i=1}^{N_{2}} C^i \sigma\left( \frac{1}{N_{1}} \sum_{j=1}^{N_{1}}W^{2,i,j}\sigma\left(W^{1,j}\cdot x\right)\right).\label{Eq:2DNN}
\end{eqnarray}

As we will see in Section \ref{SS:Extension}, the limit procedure can be extended to neural networks with three layers and subsequently to neural networks with any fixed number of hidden layers.

Notice now that (\ref{Eq:2DNN})  can be also written as
\begin{eqnarray}
H^{1,j}(x) &=& \sigma( W^{1,j}\cdot x), \phantom{.....} j =1, \ldots, N_1, \notag \\
Z^{2,i}(x) &=& \frac{1}{N_1} \sum_{j=1}^{N_1} W^{2,i,j} H^{1,j}(x), \phantom{.....} i =1, \ldots, N_2, \notag \\
H^{2,i}(x) &=& \sigma \bigg{(} Z^{2,i}(x) \bigg{)}, \phantom{.....}    \notag \\
g^{N_1, N_2}_{\theta}(x) &=& \frac{1}{N_2} \sum_{i=1}^{N_2} C^i H^{2,i}(x). \label{Eq:DeepNN}
\end{eqnarray}
where $C^i, W^{2,i,j} \in \mathbb{R}$ and $x, W^{1,j} \in \mathbb{R}^{d}$.  The neural network model has parameters
\begin{eqnarray}
\theta = (C^1, \ldots, C^{N_2}, W^{2,1,1} \ldots, W^{2,N_1,N_2}, W^{1,1},\ldots W^{1,N_1}), \notag
\end{eqnarray}
 which must be estimated from data. The number of hidden units in the first layer is $N_1$ and the number of hidden units in the second layer is $N_2$. The multi-layer neural network (\ref{Eq:DeepNN}) includes a normalization factor of $\frac{1}{N_1}$ in the first hidden layer and $\frac{1}{N_2}$ in the second hidden layer.

The  loss function that we focus on in this paper is the mean squared error
\begin{eqnarray}
L^{N_1, N_2}(\theta) =\frac{1}{2} \mathbb{E}_{Y,X} \bigg{[} ( Y - g^{N_1, N_2}_{\theta}(X)  )^2 \bigg{]},
\label{LossFunction}
\end{eqnarray}
where the data $(X,Y) \sim \pi(dx,dy)$. The goal is to estimate a set of parameters $\theta$ which minimizes the objective function (\ref{LossFunction}).

The literature frequently refers to minimizing the mean squared error loss as regression. Our results also hold for a more general class of error functions. In particular, we could have also considered the error function $\mathbb{E}_{Y,X} \Psi( Y - g^{N_1, N_2}_{\theta}(X)  ) $  for a function $\Psi$ that is smooth, convex and satisfies $\displaystyle \min_{x\in\mathbb{R}}\Psi(x)=\Psi(0)=0$. However, for the purposes of simplicity, we will focus on the standard regression task (\ref{LossFunction}).

The stochastic gradient descent (SGD) algorithm for estimating the parameters $\theta$ is, for $k \in\mathbb{N}$,
\begin{eqnarray}
C^i_{k+1} &=& C^i_{k} + \frac{\alpha_C^{N_1, N_2} }{N_2} \big{(} y_k - g^{N_1, N_2}_{\theta_k}(x_k)  \big{)} H_k^{2,i}(x_k), \notag \\
W^{1,j}_{k+1} &=& W^{1,j}_{k} +  \frac{ \alpha_{W,1}^{N_1, N_2} }{ N_1} \big{(} y_k - g^{N_1, N_2}_{\theta_k}(x_k)  \big{)} \left( \frac{1}{N_2} \sum_{i=1}^{N_2} C_k^i \sigma'(Z_k^{2, i}(x_k) )  W^{2,i,j}_k \right)  \sigma'(W^{1,j}_k \cdot x_k) x_k, \notag \\
W^{2,i,j}_{k+1} &=& W^{2,i,j}_{k}  + \frac{\alpha_{W,2}^{N_1, N_2} }{N_1 N_2}  \big{(} y_k - g^N_{\theta_k}(x_k) \big{)} C_k^i \sigma'(Z_k^{2,i}(x_k) )  H^{1,j}_k(x_k), \notag \\
H_k^{1,i}(x_k) &=& \sigma ( W^{1,i}_k  \cdot x_k), \notag \\
Z_k^{2,i}(x_k) &=& \frac{1}{N_1} \sum_{j=1}^{N_1} W^{2,i,j}_k  H^{1,j}_k(x_k), \notag \\
H_k^{2,i}(x_k) &=& \sigma( Z_k^{2,i}(x_k) ), \notag \\
g^{N_1, N_2}_{\theta_k}(x_k)  &=& \frac{1}{N_2} \sum_{i=1}^{N_2} C^i_k H^{2,i}_k(x_k).
\label{SystemDeep}
\end{eqnarray}
where $\alpha_C^{N_1, N_2}$, $\alpha_{W,1}^{N_1, N_2}$, and $\alpha_{W,2}^{N_1, N_2}$ are the learning rates. The learning rates may depend upon $N_1$ and $N_2$. The parameters at step $k$ are $\theta_k = (C^1_k, \ldots, C^{N_2}_k, W^{2,1,1}_k \ldots, W^{2,N_1,N_2}_k, W^{1,1}_k,\ldots W^{1,N_1}_k)$. $(x_k, y_k)$ are samples of the random variables $(X,Y)$.

The goal of this paper is to characterize the limit of an appropriate rescaling of the multi-layer neural network output $g^{N_1, N_2}_{\theta_k}(x)$ as both the number of hidden units $(N_1,N_2)$ and the stochastic gradient descent iterates $k$ become large. This is the topic of Theorem \ref{Theorem1}. The idea is to first take $N_1\rightarrow\infty$ with $N_2$ fixed. In Lemma \ref{LemmaFirstLayermaintext}, we prove that the empirical measure of the parameters converges to a limit measure as $N_1\rightarrow\infty$ (with $N_2$ fixed) which satisfies a measure evolution equation. This naturally implies a limit for the neural network output $g^{N_1, N_2}$ as $N_1\rightarrow\infty$. The next step is to take $N_2\rightarrow\infty$. Theorem \ref{Theorem1} proves that the limiting distribution can be represented via a system of ODEs.

As previously discussed, related limiting results for the single-layer neural network case have been investigated in \cite{Chizat2018,Szpruch2019,Montanari,Rotskoff_VandenEijnden2018,NeuralNetworkLLN,NeuralNetworkCLT}. In those papers, it is proven that as the number of hidden units and stochastic gradient descent steps, in the appropriate scaling, diverge to infinity,  the empirical distribution of the neural network parameters converges to the weak solution of a non-local PDE. This non-local PDE turns out to be a gradient flow for the limiting objective function in the space of probability measures endowed with the Wassenstein metric (this result is analogous to our Theorem \ref{T:GlobalConvergenceTheorem}). More results with non-asymptotic bounds, again for the single-layer neural network, can be found in \cite{Montanari2}.

Let us also mention that after our current paper appeared as a preprint on arXiv, \cite{Araujo2019} studied the limit of a multi-layer neural network, but under some important differences as compared to our paper. \cite{Nguyen2019} also derives asymptotics for multi-layer neural networks. In \cite{Araujo2019}, the weights in the first and last hidden layers are held fixed throughout training, and the number of hidden units in the rest of the layers is sent to infinity. In our paper, we train all parameters in all layers of the neural network, which introduces additional technical challenges. There is also the related papers \cite{Chizat2018a,NTK}, whose authors look again at the limit as the units per layer go to infinity, but under a $\frac{1}{\sqrt{N}}$ normalization instead of our $\frac{1}{N}$ normalization. In the $\frac{1}{\sqrt{N}}$ normalization, a completely different limit equation will appear. The $\frac{1}{\sqrt{N}}$ case results in a perturbation around the network's randomized initialization, leading to a kernel-type regression limit.

We address the multi-layer neural network case in the mean field scaling of $\frac{1}{N}$. In particular, we study the behavior of the neural network's output (a) as the number of hidden units in each layer go to infinity one by one, and (b) as the number of stochastic gradient descent iterates (i.e., during training of the network) goes to infinity at the speed of the number of the hidden units of the first layer. In this asymptotic regime, we are able to obtain a well-defined limit in Theorem \ref{Theorem1}. Consequences of this result along with a simulation study are presented in Sections \ref{S:GlobalConvergence} and \ref{S:Consequences}.

The rest of the paper is organized as follows. Our main result, which characterizes the asymptotic behavior of a neural network with two hidden layers when the number of hidden units becomes large, is presented in Section \ref{ResultsIntro}. The result can be easily extended to an arbitrary number of hidden layers. Section \ref{S:GlobalConvergence} is devoted to the global convergence arguments. In particular, we show that under the proper assumption the limiting problem derived in Section \ref{ResultsIntro} seeks to minimize the limiting objective function, recovering the global minimum. Section \ref{S:Consequences} discusses further the theoretical results, includes a numerical study to showcase some of the theoretical implications, and, as an example, presents the limit for a three-layer neural network. The proof of the convergence theorem is in Section \ref{ProofSecondLayer}. The uniqueness  of a solution to the limiting system is established in Section \ref{S:PropertiesLimitingSystem}. The proof of the limit of the first layer (i.e., the proof of Lemma \ref{LemmaFirstLayermaintext}) and a few other results are  provided in the Appendix. Section \ref{S:Conclusions} has concluding remarks and possible directions for future work.

\section{Main Results} \label{ResultsIntro}

Let us start by presenting our assumptions, which will hold throughout the paper. We shall work on a filtered probability space $(\Omega,\mathcal{F},\mathbb{P})$  on which all the random variables are defined. The probability space is equipped with a filtration $\mathfrak{F}_t$ that is right continuous and $\mathfrak{F}_0$ contains all $\mathbb{P}$-negligible sets.

\begin{assumption}\label{A:Assumption1}
We assume the following conditions throughout the paper.
  \begin{itemize}
\item $\sigma(\cdot) \in C^{2}_b$, i.e., it is twice continuously differentiable and bounded.
\item The distribution $\pi(dx,dy)$ has compact support, i.e. the data $(x_k, y_k)$ takes values in the compact set $\mathcal{X} \times \mathcal{Y}$.
\item The random initialization of the parameters, i.e. $\{C^i_{\circ}\}_{i}, \{W^{2,i,j}_{\circ}\}_{i,j} , \{W_{\circ}^{1,j}\}_{j}$, are i.i.d. and take values in compact sets $\mathcal{C}, \mathcal{W}^1$, and $\mathcal{W}^{2}$.
\item The probability distributions of the initial parameters $( C^{i}_{\circ}, W^{2,i,j}_{\circ} , W^{1,j}_{\circ} )_{i,j}$ admit continuous probability density functions.
\end{itemize}
\end{assumption}

We denote by $\mu_c(dc)$, $\mu_{W^2}(du)$, and $\mu_{W^{1}}(dw)$  the probability distributions of $\{C^i_{\circ}\}_{i}$, $\{W^{2,i,j}_{\circ}\}_{i,j}$, and $\{W_{\circ}^{1,j}\}_{j}$ respectively.

For reasons that will become clearer later on, we shall choose the learning rates to be
\begin{align}
\alpha_C^{N_1, N_2} &= \frac{N_2}{N_1},\quad \alpha_{W,1}^{N_1, N_2} = 1 \text{ and }\alpha_{W,2}^{N_1,N_2} = N_2.\label{Eq:LearningRate}
\end{align}

Note that the weights in the second layer are trained faster than the other parameters. This choice of learning rates is necessary for convergence to a non-trivial limit as $N_1, N_2 \rightarrow \infty$. If the parameters in all the layers are trained with the same learning rate, it can be mathematically shown that the network will not train as $N_1, N_2$ become large. We further explore this interesting fact in Section \ref{SS:Disucssion}.

Define the empirical measure
\begin{eqnarray}
\tilde \gamma_k^{N_1, N_2} \vcentcolon= \frac{1}{N_1} \sum_{j=1}^{N_1} \delta_{ W_k^{1,j}, W^{2,1,j}_{k}, \ldots, W^{2,N_2, j}_{k}, C_k^1, \ldots, C_k^{N_2}}.\label{Eq:EmpiricalMeasureMultilayer}
\end{eqnarray}

If $f$ is an appropriate test function on some space $X$ and $\gamma$ is a finite measure on $X$, then we denote the inner product $\la f, \gamma \ra=\int_{X}f(x)\gamma (dx)$. Using this inner product, the neural network's output can be re-written in terms of the empirical measure:
\begin{eqnarray}
g_{\theta_k}^{N_1, N_2}(x)  = \frac{1}{N_2}  \sum_{i=1}^{N_2} \la c_i, \tilde  \gamma^{N_1, N_2}_k \ra  \sigma \bigg{(}  \la w^{2,i} \sigma(w^{1} \cdot x) , \tilde \gamma^{N_1, N_2}_k \ra \bigg{)}.\nonumber
\end{eqnarray}

Let us next define the time-scaled empirical measure
\begin{eqnarray}
\gamma_t^{N_1, N_2} \vcentcolon= \tilde \gamma_{\floor*{N_{1} t} }^{N_1, N_2},\nonumber
\end{eqnarray}

and the corresponding time-scaled neural network output is
\begin{eqnarray}
g_{t}^{N_1, N_2}(x) \vcentcolon = g_{\theta_{ \floor*{N_{1} t} }}^{N_1, N_2}(x) = \frac{1}{N_2}  \sum_{i=1}^{N_2} \la c_i,  \gamma^{N_1, N_2}_t \ra  \sigma \bigg{(}  \la w^{2,i} \sigma(w^{1} \cdot x) ,  \gamma^{N_1, N_2}_t \ra \bigg{)}.\nonumber
\end{eqnarray}

At any time $t$, $\gamma^{N_1, N_2}_t$ is measure-valued.  The scaled empirical measure $( \gamma^{N_1, N_2}_t)_{ 0 \leq t \leq 1}$ is a random element of
$D_{E}([0,1])$\footnote{$D_S([0,1])$ is the set of maps from $[0,1]$ into $S$ which are right-continuous and which have left-hand limits.} with $E = \mathcal{M}(\mathbb{R}^{d + 2 N_2})$.

We study convergence using iterated limits. We first let $N_1 \rightarrow \infty$ where the number of units in the first layer is $N_1$ and the number of stochastic gradient descent steps is $\floor*{N_1}$.
Then, we let the number of units in the second layer $N_2 \rightarrow \infty$.

We begin by letting the number of hidden units in the first layer $N_1 \rightarrow \infty$.
\begin{lemma} \label{LemmaFirstLayermaintext}
The process $\gamma^{N_1, N_2} \vcentcolon = ( \gamma_t^{N_1, N_2})_{0 \leq t \leq 1}$ converges in distribution to the measure valued process  $\gamma^{N_2}$ that takes values in $D_E([0,1])$ as $N_1 \rightarrow \infty$. For every $f\in C^{2}_{b}(\mathbb{R}^{d + 2 N_2})$, $\gamma^{N_2}$ satisfies the measure evolution equation
\begin{align}
\la f, \gamma^{N_2}_t \ra - \la f, \gamma^{N_2}_0 \ra &= \int_0^t   \int_{\mathcal{X}\times\mathcal{Y}}   \big{(} y -  g_s^{N_2}(x)   \big{)} \la H_s^{N_2}(x) \cdot \nabla_c f, \gamma^{N_2}_{s} \ra \pi(dx,dy)  ds\notag \\
&+ \int_0^t \int_{\mathcal{X}\times\mathcal{Y}}   \big{(} y -  g_s^{N_2}(x) \big{)} \la  \sigma(w^{1} \cdot x) (\sigma'(Z_{s}(x)) \odot c) \cdot \nabla_{w^2} f, \gamma^{N_2}_{s} \ra \pi(dx, dy) ds \notag \\
&+ \int_0^t \int_{\mathcal{X}\times\mathcal{Y}}    \big{(} y -  g_s^{N_2}(x) \big{)} \frac{1}{N_2} \sum_{i=1}^{N_2} \la  c_i \sigma'(Z_s^{i, N_2}(x)) w^{2,i} \sigma'( w^{1} \cdot x) x \cdot \nabla_{w^{1}} f, \gamma^{N_2}_{s} \ra \pi(dx, dy) ds,
\label{LimitSPDELLNmaintext}
\end{align}
where
\begin{align}
Z_s^{i, N_2}(x) &= \la w^{2,i} \sigma(w^{1} \cdot x), \gamma_s^{N_2} \ra, \notag \\
H_s^{i, N_2}(x) &=  \sigma( Z_s^{i, N_2}(x) ), \notag \\
g_s^{N_2}(x) &= \frac{1}{N_2}  \sum_{i=1}^{N_2} \la c_i, \gamma_s^{N_2} \ra H_s^{i, N_2}(x) , \notag \\
\gamma^{N_2}_0(dw^{1}, dw^{2}, dc)  &= \mu_{W^{1}}(dw^{1}) \times \mu_{W^2}(dw^{2,1}) \times \cdots \times \mu_{W^2}( dw^{2,N_2} ) \times \delta_{ C_{\circ}^1}(dc^1) \times \cdots \times   \delta_{ C_{\circ}^{N_2}}(dc^{N_2}).\label{LimitSPDELLNmaintextb}
\end{align}
\end{lemma}
\begin{proof}
The proof of this lemma is related to the limit in the first layer as the number of hidden units in the first layer grows with the number of hidden units in the second layer is held fixed. The proof is analogous to the proof in \cite{NeuralNetworkLLN} and the details are presented for completeness in the Appendix \ref{ProofFirstLayer}.
\end{proof}

Lemma \ref{LemmaFirstLayermaintext} studies the limit of the empirical measure $\gamma_{t}^{N_1,N_2}$ as $N_1\rightarrow \infty$ with $N_2$ fixed. The limit is characterized by the stochastic evolution equation (\ref{LimitSPDELLNmaintext})-(\ref{LimitSPDELLNmaintextb}). Notice that Lemma \ref{LemmaFirstLayermaintext} immediately implies that
\[
\lim_{N_1\rightarrow\infty} g_{t}^{N_1, N_2}(x) = g_t^{N_2}(x),
\]
in probability, as $N_1 \rightarrow \infty$

The next step is to study the limit as $N_2\rightarrow \infty$. To do so, we study the limit of the  random ODE as $N_2\rightarrow\infty$ whose law is characterized by (\ref{LimitSPDELLNmaintext})-(\ref{LimitSPDELLNmaintextb}). Our main goal is the characterization of the limit neural network output $g_{t}^{N_1, N_2}(x)$. The following convergence result characterizes the neural network output $g_{t}^{N_1, N_2}(x)$ for large $N_1$ and $N_2$.

\begin{theorem} \label{Theorem1}
For any $t \in [0,1]$ and $x \in \mathcal{X}$,
\begin{eqnarray}
&\phantom{.}& \lim_{N_2 \rightarrow \infty} \lim_{N_1 \rightarrow \infty} g_{t}^{N_1, N_2}(x)   = g_t(x), \notag
\end{eqnarray}
in probability\footnote{$\displaystyle \lim_{N_2 \rightarrow \infty} \lim_{N_1 \rightarrow \infty}  X^{N_1, N_2} = X$ in probability if, for all $\epsilon > 0$, $\displaystyle \lim_{N_2 \rightarrow \infty} \lim_{N_1 \rightarrow \infty} \mathbb{P} \bigg{[} \norm{X^{N_1, N_2} - X }  > \epsilon \bigg{]} = 0$.}, where we have that
\begin{align}
g_t(x) &= \int_{\mathcal{C}}  \tilde C_t^{c} \tilde H_t^{2,c}(x) \mu_c(dc), \label{Eq:LLN_DNNTheoremIntro0}
\end{align}
with
\begin{eqnarray}
d \tilde C_t^{c} &=&  \int_{\mathcal{X} \times \mathcal{Y} } \big{(} y -  g_t(x)  \big{)} \tilde H_t^{2,c}(x) \pi(dx,dy)  dt, \phantom{....} 
\tilde C_0^{c} = c, \notag \\
 d \tilde W_t^{1,w} &=&  \int_{\mathcal{X} \times \mathcal{Y} } \big{(} y - g_t(x)  \big{)} V_t^{w}(x) \sigma'( \tilde W_t^{1,w} \cdot x) x \pi(dx,dy)   dt, \phantom{....}
 \tilde W_0^{1,w} = w, \notag \\
d \tilde W^{2,c,w,u}_{t} &=&  \int_{\mathcal{X} \times \mathcal{Y} } \big{(} y - g_t(x) \big{)} \tilde C_t^{c} \sigma'(\tilde Z_t^{c}(x) )  \tilde H_t^{1,w}(x) \pi(dx,dy)  dt,  \phantom{......} \tilde W^{2,c,w,u}_0 = u,  \notag \\
\tilde H_t^{1,w}(x) &=& \sigma ( \tilde W_t^{1,w}\cdot  x), \notag \\
\tilde Z_t^{c}(x) &=&  \int_{\mathcal{W}^{1}} \int_{\mathcal{W}^2}   \tilde  W^{2,c,w,u }_t  \tilde H^{1,w}_t(x)  \mu_{W^2}(du) \mu_{W^{1}}(dw) , \notag \\
\tilde H_t^{2,c}(x) &=& \sigma( \tilde Z_t^{c}(x) ), \notag \\
V_t^{w}(x) &=& \int_{\mathcal{C}}  \tilde C_t^c  \sigma'(\tilde Z_t^{c}(x) ) \left( \int_{\mathcal{W}^2}  \tilde W^{2,c,w,u}_t   \mu_{W^2}(du) \right) \mu_c(dc). \label{Eq:LLN_DNNTheoremIntro}
\end{eqnarray}

The system in (\ref{Eq:LLN_DNNTheoremIntro}) has a unique solution. In addition, letting $g_t^{N_2}(x)$ defined through Lemma \ref{LemmaFirstLayermaintext} we have the following rate of convergence
\begin{eqnarray}
\sup_{x \in \mathcal{X}} \mathbb{E} \bigg{[} | g_t^{N_2}(x) - g_t(x) | \bigg{]} \leq K N_2^{-1/2},\nonumber
\end{eqnarray}
for some constant $K<\infty$.
\end{theorem}

Notice that we can also write that $g_t(x)$ satisfies
\begin{align}
g_t(x) &= \int_{\mathcal{C}}  \tilde C_t^{c} \sigma\left( \int_{\mathcal{W}^{1}} \int_{\mathcal{W}^2}   \tilde  W^{2,c,w,u }_t  \sigma ( \tilde W_t^{1,w}\cdot  x)  \mu_{W^2}(du) \mu_{W^{1}}(dw) \right) \mu_c(dc).\label{Eq:CompressedNNoutput}
\end{align}

The proof of Theorem \ref{Theorem1} is given in Section \ref{ProofSecondLayer}. Theorem \ref{Theorem1} indicates that the neural network output in the large hidden units and large SGD-iterates limit depends on paths of representative weights that connect the input layer to the output layer. Even though, we restricted the statement of Theorem \ref{Theorem1} in the interval $t\in[0,1]$, its proof makes it clear that the statement is true for $t\in[0,T]$ for any $0<T<\infty$. In addition, even though this does not mean that we can take $T=\infty$, we can still examine what happens to the limit problem as $t$ grows.  In particular, in Section \ref{S:GlobalConvergence} we show that under the proper assumptions, one does expect to recover the global minimum as $t\rightarrow\infty$.

Section \ref{S:Consequences} discusses some further consequences of Theorem \ref{Theorem1} as well as challenges that come up in the study of the limiting behavior of multi-layer neural networks as the number of the hidden units grows.

\section{On global convergence}\label{S:GlobalConvergence}
The goal of this section is to demonstrate that, under appropriate assumptions, it can be expected that the global minimum is recovered as $t\rightarrow\infty$. Let the target data $y$ be a function $f(x)$ that we seek to learn, i.e. $y = f(x)$.

\begin{assumption} \label{A:AssumptionJS}
The activation function $\sigma(\cdot)$ is real analytic, bounded, and $\sigma'(\cdot) > 0$. Let $\mathcal{X} = \{1 \} \times D$ where $D \subset \mathbb{R}^{d-1}$, which is equivalent to including bias weights in the first layer of the neural network. $\pi$ is positive on sets in $D$ with positive Lebesgue measure and its first marginal is a Dirac measure at $x=1$.

\end{assumption}

Notice that by \cite{Hornik2}, the fact that $\sigma$ is bounded and non-constant by Assumption \ref{A:AssumptionJS} implies that $\sigma$ is also discriminatory in the sense of \cite{Cybenk1989,Hornik2}. Namely, if we have that
\[
\int_{\mathcal{X}}h(x)\sigma(w\cdot x)\pi(dx)=0 \text{ for all } w \in \mathbb{R}^d,
\]
then $h(x)=0$ for $x \in \mathcal{X}$ .

\begin{remark}\label{R:ExampleActivationFcn}
An example of a real analytic, bounded activation function where $\sigma'(\cdot) > 0$ (and therefore discriminatory too), i.e. that satisfies Assumption  \ref{A:AssumptionJS}, is the sigmoid function $\sigma(z) = \frac{e^z}{1+ e^z}$.
\end{remark}

Theorem 2 of \cite{Funakashi1989} shows that for activation functions $\sigma$ that are non-constant, bounded and monotone (in which case Assumption  \ref{A:AssumptionJS} holds), then two layer neural networks, such as $g_{\theta}^{N_1, N_2}(x)$, (and more generally multilayer neural networks) are dense in $\mathcal{C}(\mathcal{X})$. This result basically implies that for large enough $N_1,N_2$ functions of the form $g_{\theta}^{N_1, N_2}(x)$ approximate to arbitrary accuracy functions in $\mathcal{C}(\mathcal{X})$. Training the weights with stochastic grading descent and sending $N_1,N_2$ to infinity, leads then, by Theorem \ref{Theorem1}, to function approximator $g_{t}(x)$, which is given through a system of integro-differential equations of the form
\begin{align}
d \tilde C_t^{c} &=  \int_{\mathcal{X} } \big{(} f(x) -  g_t(x)  \big{)} \tilde H_t^{2,c}(x) \pi(dx)  dt, \phantom{....} 
\tilde C_0^{c} = c, \notag \\
 d \tilde W_t^{1,w} &=  \int_{\mathcal{X} } \big{(} f(x) - g_t(x)  \big{)} V_t^{w} \sigma'( \tilde W_t^{1,w} \cdot x) x \pi(dx)   dt, \phantom{....}
 \tilde W_0^{1,w} = w, \notag \\
d \tilde W^{2,c,w,u}_{t} &=  \int_{\mathcal{X}  } \big{(} f(x) - g_t(x) \big{)} \tilde C_t^{c} \sigma'(\tilde Z_t^{c} )  \tilde H_t^{1,w}(x) \pi(dx)  dt,  \phantom{......} \tilde W^{2,c,w,u}_0 = u,  \notag \\
\tilde H_t^{1,w}(x) &= \sigma ( \tilde W_t^{1,w}\cdot  x), \notag \\
\tilde Z_t^{c}(x) &=  \int_{\mathcal{W}^{1}} \int_{\mathcal{W}^2}   \tilde  W^{2,c,w,u }_t  \tilde H^{1,w}_t(x)  \mu_{W^2}(du) \mu_{W^{1}}(dw) , \notag \\
\tilde H_t^{2,c}(x) &= \sigma( \tilde Z_t^{c}(x) ), \notag \\
V_t^{w}(x) &= \int_{\mathcal{C}}  \tilde C_t^c  \sigma'(\tilde Z_t^{c}(x) ) \left( \int_{\mathcal{W}^2}  \tilde W^{2,c,w,u}_t   \mu_{W^2}(du) \right) \mu_c(dc),\label{Eq:LLN_DNNTheoremIntro00}
\end{align}
where the neural network prediction is
\begin{align}
g_t(x) &= \int_{\mathcal{C}}  \tilde C_t^{c} \tilde H_t^{2,c}(x) \mu_c(dc), \label{Eq:LLN_DNNTheoremIntro000}
\end{align}

We establish in this section that under reasonable assumptions $f(x)$ can indeed be recovered  by $g_{t}(x)$ as $t\rightarrow\infty$.

Let us denote $\Theta_{t}(c,w,u)=(\tilde{C}^{c}_{t},\tilde{W}^{1,w}_{t},\tilde{W}^{2,c,w,u}_{t})$ for the components of the ODE in (\ref{Eq:LLN_DNNTheoremIntro00}). For notational convenience, we shall often write $\theta=(c,w,u)$. Then, we obviously have that $g_{t}(x)$ depends on $t$ only through $[\Theta_t]=\left\{\Theta_{t}(\theta)\right\}_{\theta\in \mathcal{C}\times\mathcal{W}^{1}\times\mathcal{W}^{2}}$. In order to emphasize that we shall write $g_{t}(x)=g(x;[\Theta_{t}])$.

Analogously, we will denote $[\Theta_t(c,\cdot)]=\left\{\Theta_{t}(c,w,u)\right\}_{(w,u)\in \mathcal{W}^{1}\times\mathcal{W}^{2}}$ and
$[\Theta_t(\cdot,w,\cdot)]=\left\{\Theta_{t}(c,w,u)\right\}_{(c,u)\in \mathcal{C}\times\mathcal{W}^{2}}$, which then leads to the notation $\tilde{Z}^{c}_{t}(x)=\tilde{Z}(x;[\Theta_{t}(c,\cdot,\cdot)])$ and $V^{w}_{t}(x)=V(x;[\Theta_{t}(\cdot,w,\cdot)])$. For notational convenience let us also denote $h(x;[\Theta_{t}])=(f(x)-g_{t}(x))$.

With these definitions in place, and for $(c,w)\in\mathcal{C}\times\mathcal{W}^{1}$, let us also define
\begin{align}
R_{1}([\Theta_t],c)&\equiv \int_{\mathcal{X}}   h(x;[\Theta_t])     \sigma(\tilde{Z}(x;[\Theta_t(c,\cdot)]))\pi(dx)\nonumber\\
R_{2}([\Theta_t],\tilde W^{1,w}_t, w)&\equiv \int_{\mathcal{X}} h(x;[\Theta_t]) V(x;[\Theta_t(\cdot,w,\cdot)]) \sigma'( \tilde W^{1,w}_t \cdot x)  x  \pi(dx)\nonumber\\
R_{3}([\Theta_t],\tilde C^{c}_t,\tilde W^{1,w}_t, c)&\equiv \int_{\mathcal{X}} h(x;[\Theta_t])    \tilde C^{c}_t\sigma'(\tilde{Z}(x;[\Theta_t(c,\cdot)])))     \sigma ( \tilde W^{1,w}_t\cdot  x)\pi(dx)\label{Eq:ODEdefinition1}
\end{align}

and let us set
\[
H([\Theta_t],\Theta_t(\theta), \theta)= (R_{1}([\Theta_t],c), R_{2}([\Theta_t],\tilde W^{1,w}_t, w),R_{3}([\Theta_t],\tilde C^{c}_t,\tilde W^{1,w}_t, c)).
\]

The purpose of the notation above is to make it clear that the functions $H(\cdot)$ depends on $[\Theta_t]$, on $\Theta_t(\theta)$ and on  $\theta$ separately. With these identifications we can write that the ODE system in (\ref{Eq:LLN_DNNTheoremIntro00}) can be written in the form
\begin{align}
\dot{\Theta}_{t}(\Theta_{0})&= H([\Theta_t],\Theta_t(\Theta_{0}), \Theta_{0}), \textrm{ such that } \Theta_{0}=(c,w,u). \label{Eq:ODEdefinition2}
\end{align}

Next we investigate the limiting objective function
\begin{align}
\bar{L}(\Theta_{t})&=\lim_{N_2\rightarrow\infty}\lim_{N_1\rightarrow\infty}L^{N_{1},N_{2}}(\theta_{\floor*{N_{1}t}}) =\lim_{N_2\rightarrow\infty}\lim_{N_1\rightarrow\infty}\frac{1}{2} \mathbb{E}_{X} \bigg{[} ( f(X)-g^{N_1, N_2}_{\theta_{\floor*{N_{1}t}}}(X)  )^2 \bigg{]}\nonumber\\
&= \frac{1}{2} \int_{\mathcal{X}} \bigg{[} ( f(x)- g(x;[\Theta_{t}])  )^2 \bigg{]}\pi(dx)\nonumber
\end{align}
where $g(x;[\Theta_{t}])=g_{t}(x) $ is given by (\ref{Eq:LLN_DNNTheoremIntro000}).

Notice that the function $\Theta\mapsto \bar{L}(\Theta)$ is always non-negative and becomes zero only at a global  minimum. In particular, our goal is to demonstrate that as $t\rightarrow\infty$, there are sufficient conditions that guarantee that the global minimum is realized in the sense that
\[
\lim_{t\rightarrow\infty}g(x;[\Theta_{t}])=f(x), \text{ for almost all }x\in\mathcal{X}.
\]

We notice that $\bar{L}$ acts as a Lyapunov function for the dynamical system (\ref{Eq:ODEdefinition2}) in that
\begin{align*}
\frac{d}{dt}\bar{L}(\Theta_{t})&=\nabla_{\Theta} \bar{L}(\Theta_{t})\cdot \dot{\Theta}_{t}= \nabla_{\Theta} \bar{L}(\Theta_{t})\cdot  H(\Theta_{t})\leq 0
\end{align*}

Indeed, we calculate
\begin{align}
&\frac{d}{dt}\bar{L}(\Theta_{t})=\int_{\mathcal{X}}  ( f(x)- g_{t}(x)  ) \dot{g}_{t}(x)\pi(dx)\nonumber\\
&=\int_{\mathcal{X}}  ( f(x)- g_{t}(x)  ) \left[\int_{\mathcal{C}}  \left(\dot{\tilde C}_t^{c} \tilde H_t^{2,c}(x)+\right.\right.\nonumber\\
&\quad\left.\left.+
 \tilde C_t^{c}
\sigma'(\tilde Z_t^{c}(x) )\left(\int_{\mathcal{W}^{1}} \int_{\mathcal{W}^2}   \left(\dot{\tilde  W}^{2,c,w,u }_t  \tilde H^{1,w}_t(x)+\tilde  W^{2,c,w,u }_t \sigma'( \tilde W_t^{1,w} \cdot x) \dot{\tilde W}_t^{1,w} \cdot x\right)  \mu_{W^2}(du) \mu_{W^{1}}(dw)\right)
\right)    \mu_c(dc)\right]\pi(dx)\nonumber\\
&=\int_{\mathcal{X}}  ( f(x)- g_{t}(x)  ) \int_{\mathcal{C}} \dot{\tilde C}_t^{c} \tilde H_t^{2,c}(x)\mu_c(dc)\pi(dx)\nonumber\\
&\quad+\int_{\mathcal{X}}  ( f(x)- g_{t}(x)  ) \int_{\mathcal{C}} \tilde C_t^{c}
\sigma'(\tilde Z_t^{c}(x) )\int_{\mathcal{W}^{1}} \int_{\mathcal{W}^2}   \tilde  W^{2,c,w,u }_t \sigma'( \tilde W_t^{1,w} \cdot x) \dot{\tilde W}_t^{1,w} \cdot x  \mu_{W^2}(du) \mu_{W^{1}}(dw)    \mu_c(dc)\pi(dx)\nonumber\\
&\quad+\int_{\mathcal{X}}  ( f(x)- g_{t}(x)  ) \int_{\mathcal{C}} \tilde C_t^{c}
\sigma'(\tilde Z_t^{c}(x) )\int_{\mathcal{W}^{1}} \int_{\mathcal{W}^2}   \dot{\tilde  W}^{2,c,w,u }_t  \tilde H^{1,w}_t(x)\mu_{W^2}(du) \mu_{W^{1}}(dw)\mu_c(dc)\pi(dx)\nonumber
\end{align}
\begin{align}
&=\int_{\mathcal{C}} \dot{\tilde C}_t^{c} \left[\int_{\mathcal{X}}  ( f(x)- g_{t}(x)  )   \tilde H_t^{2,c}(x)\pi(dx)\right]\mu_c(dc)\nonumber\\
&\quad+\int_{\mathcal{W}^{1}}  \dot{\tilde W}_t^{1,w} \cdot \left[\int_{\mathcal{X}}  ( f(x)- g_{t}(x)  ) \int_{\mathcal{C}} \tilde C_t^{c}
\sigma'(\tilde Z_t^{c}(x) ) \int_{\mathcal{W}^2}   \tilde  W^{2,c,w,u }_t \sigma'( \tilde W_t^{1,w} \cdot x)  x  \mu_{W^2}(du) \mu_c(dc)\pi(dx)\right]\mu_{W^{1}}(dw)\nonumber\\
&\quad+\int_{\mathcal{C}}\int_{\mathcal{W}^{1}} \int_{\mathcal{W}^2}\dot{\tilde  W}^{2,c,w,u }_t \left[\int_{\mathcal{X}}  ( f(x)- g_{t}(x)  )  \tilde C_t^{c}
\sigma'(\tilde Z_t^{c}(x) )     \tilde H^{1,w}_t(x)\pi(dx)\right]\mu_{W^2}(du) \mu_{W^{1}}(dw)\mu_c(dc)\nonumber\\
&=\int_{\mathcal{C}} \dot{\tilde C}_t^{c} \left[\int_{\mathcal{X}}  ( f(x)- g_{t}(x)  )   \tilde H_t^{2,c}(x)\pi(dx)\right]\mu_c(dc)\nonumber\\
&\quad+\int_{\mathcal{W}^{1}}  \dot{\tilde W}_t^{1,w} \cdot \left[\int_{\mathcal{X}}  ( f(x)- g_{t}(x)  ) V_t^{w}
 \sigma'( \tilde W_t^{1,w} \cdot x)  x  \pi(dx)\right]\mu_{W^{1}}(dw)\nonumber\\
 &\quad+\int_{\mathcal{C}}\int_{\mathcal{W}^{1}} \int_{\mathcal{W}^2}\dot{\tilde  W}^{2,c,w,u }_t \left[\int_{\mathcal{X}}  ( f(x)- g_{t}(x)  )  \tilde C_t^{c}
\sigma'(\tilde Z_t^{c}(x) )     \tilde H^{1,w}_t(x)\pi(dx)\right]\mu_{W^2}(du) \mu_{W^{1}}(dw)\mu_c(dc)\nonumber\\
&=-\int_{\mathcal{C}}  \left(\int_{\mathcal{X}}  ( f(x)- g_{t}(x)  )   \tilde H_t^{2,c}(x)\pi(dx)\right)^{2}\mu_c(dc)\nonumber\\
&\quad-\int_{\mathcal{W}^{1}}   \left|\int_{\mathcal{X}}  ( f(x)- g_{t}(x)  ) V_t^{w}
 \sigma'( \tilde W_t^{1,w} \cdot x)  x  \pi(dx)\right|^{2}\mu_{W^{1}}(dw)\nonumber\\
 &\quad-\int_{\mathcal{C}}\int_{\mathcal{W}^{1}} \int_{\mathcal{W}^2} \left(\int_{\mathcal{X}}  ( f(x)- g_{t}(x)  )  \tilde C_t^{c}
\sigma'(\tilde Z_t^{c}(x) )     \tilde H^{1,w}_t(x)\pi(dx)\right)^{2}\mu_{W^2}(du) \mu_{W^{1}}(dw)\mu_c(dc)\nonumber
\end{align}
\begin{align}
&=-\left[\int_{\mathcal{C}}  \left(R_{1}([\Theta_t],c)\right)^{2}\mu_c(dc)+\int_{\mathcal{W}^{1}}   \left|R_{2}([\Theta_t],\tilde W^{1,w}_t, w)\right|^{2}\mu_{W^{1}}(dw)\right.\nonumber\\
&\qquad\left.+\int_{\mathcal{C}}\int_{\mathcal{W}^{1}} \left(R_{3}([\Theta_t],\tilde C^{c}_t,\tilde W^{1,w}_t, c)\right)^{2}\mu_{W^{1}}(dw)\mu_c(dc)\right]\nonumber\\
&\leq 0.\label{Eq:LimitLossDerivative}
\end{align}

The fact that $\frac{d}{dt}\bar{L}(\Theta_{t})\leq 0$ from (\ref{Eq:LimitLossDerivative}) means that $\bar{L}(\Theta_t)$ is decreasing in the gradient direction of the paths governing the limiting behavior of the weights.

Let us define $\eta_{0}$ to be the joint measure for the random initialization of the parameters, i.e. for $\{C^i_{\circ}\}_{i}$, $\{W_{\circ}^{1,j}\}_{j}$, $\{W^{2,i,j}_{\circ}\}_{i,j}$. By Assumption \ref{A:Assumption1} we have that this is the product measure $\eta_{0}=\mu_{c}\times\mu_{W^{1}}\times\mu_{W^{2}}$. Let us make a convenient assumption for the support of $\eta_{0}$. In particular,
\begin{assumption}\label{A:SupportInitial}
We have that $\textrm{support}(\eta_{0})=\mathcal{C}\times\mathcal{W}^{1}\times\mathcal{W}^{2}$. \end{assumption}

Notice now that we can write
\begin{align}
\frac{d}{dt}\bar{L}(\Theta_{t})&=
-\int_{\mathcal{C}}\int_{\mathcal{W}^{1}}\int_{\mathcal{W}^{2}}\left|H([\Theta_t],\Theta_t(\theta), \theta)\right|^{2}\eta_{0}(d\theta)\leq 0.\label{Eq:LimitLossDerivative2}
\end{align}

    As a matter of fact, more is true. We investigate how the objective function changes with perturbation from $\Theta$ to $\hat{\Theta}$. In particular, we have
for $\epsilon=(\epsilon_{1},\epsilon_{2},\epsilon_{3})$
\begin{align*}
\frac{\partial}{\partial \epsilon_{i}}\bar{L}(\Theta+\epsilon\odot\hat{\Theta})\Big|_{\epsilon=0}&=\int_{\mathcal{C}\times\mathcal{W}^{1}\times\mathcal{W}^{2}}H_{i}([\Theta],\Theta(\theta),\theta)\hat{\Theta}_{i}(\theta)\eta_{0}(d\theta),
\end{align*}
which, in combination with (\ref{Eq:LimitLossDerivative2}) and Assumption \ref{A:SupportInitial}, then says that  a minimizer for $\bar{L}(\Theta)$, say  $\Theta^{*}(\theta)$, will satisfy for almost all $(c,w)\in\mathcal{C}\times\mathcal{W}^{1}$ the relations
\begin{align}
R_{1}([\Theta^{*}],c)=R_{2}([\Theta^{*}],\tilde W^{1,*,w}, w)=R_{3}([\Theta^{*}],\tilde C^{*,c},\tilde W^{1,*,w}, c)=0.
\label{Eq:Minimizers}
\end{align}

By analogy, let us now define $\eta_{t}$ to be the measure  at time $t$  of the random vector $\Theta_{t}=(\tilde{C}_{t},\tilde{W}^{1}_{t},\tilde{W}^{2}_{t})$ as governed by the solution to the random ODE system (\ref{Eq:LLN_DNNTheoremIntro00}). Then,  $\eta_{t}$ is the pushforward of $\eta_{0}$ under $\Theta_{t}$ given by (\ref{Eq:ODEdefinition2}) (see for example Chapter 8 in \cite{Ambrosio2008}), i.e.,
\begin{align}
\eta_{t}&=\left(\Theta_{t}\right)_\sharp \eta_{0}.\label{Eq:PushforwardMeasure0}
\end{align}

\begin{theorem}\label{T:GlobalConvergenceTheorem}
Let assumptions  \ref{A:AssumptionJS} and \ref{A:SupportInitial} hold. If $\eta_{t}\rightarrow \eta^{*}$, weakly, where $\eta^{*}$ is a non-degenerate measure  that admits a density with finite first moments, then we have that $\eta^{\ast}$ is a global minimum with zero loss.
\end{theorem}

Before proving Theorem \ref{T:GlobalConvergenceTheorem} we discuss its assumptions in the remarks that follow.
\begin{remark}\label{R:RemarkForMainThm}
It is interesting to note that Theorem \ref{T:GlobalConvergenceTheorem} actually  shows that any stationary point is a global minimum with zero loss. In addition, even though we do not show this here, one can see from the proof of Theorem \ref{T:GlobalConvergenceTheorem} that the assumption on analyticity of the activation function can be replaced by an assumption on $\eta^{*}$ having full support.
\end{remark}

\begin{remark}\label{R:ConvergenceODEs}
The assumption on convergence of $\eta_t$ as $t\rightarrow\infty$ is a required assumption for convergence in the limit as $t\rightarrow\infty$ to occur.  For completeness we mention here that verifying this assumption for the (commonly used in practice) unregularized case that we study here is still an open problem even for the corresponding unregularized  case in the single layer case, see Appendix C.5 in \cite{Chizat2018} or \cite{Montanari}. Even though we do not show this here, one way to guarantee that such an assumption  holds is to add an appropriate regularization term in the loss function,  or to add an appropriate non-degenerate stochastic perturbation in the training algorithm (\ref{SystemDeep}) for the evolution of each one of $C^i_{k+1},W^{1,j}_{k+1}, W^{2,i,j}_{k+1}$ (noisy SGD); see \cite{Szpruch2019,Rotskoff_VandenEijnden2018} for the corresponding arguments in the single layer case. See also \cite{Chizat2019} for conditions under which the assumption on convergence of $\eta_t$ as $t\rightarrow\infty$ holds for non-degenerate sparse problems that include an appropriate regularizing term in the loss function.
\end{remark}

\begin{remark}\label{R:LiftingSpaceOfMeasures}
It is clear that convergence of $[\Theta_{t}]$ as $t\rightarrow\infty$ is linked to the convergence of the measure $\eta_{t}$ as $t\rightarrow\infty$. At this point, we also mention that convergence of the trajectories $[\Theta_{t}]$ as $t\rightarrow\infty$ is linked to the study of {\L}ojasiewicz inequalities \cite{Bolte2017,Mazon2019}. To the best of our knowledge no general results in this direction are currently known neither for our case of interest, nor for problems corresponding to the single-layer case (see Appendix C.5 of \cite{Chizat2018}) (i.e. no general results currently exist for the non-geodesically convex case that we are dealing with here).
\end{remark}

\begin{proof}[Proof of Theorem \ref{T:GlobalConvergenceTheorem}]
We will use the notation $\tilde{\mu}_c(c)$, $\tilde{\mu}_{W^1}(w)$, and $\tilde{\mu}_{W^2}(w)$ for the densities of the measure $\mu_c(dc)$, $\mu_{W^1}(dw)$, and $\mu_{W^2}(dw)$ respectively. From equations (\ref{Eq:Minimizers}), we have that any stationary solution must satisfy
\begin{align}
0 &= \int_{\mathcal{C}}\int_{\mathcal{W}^{1}} \int_{\mathcal{W}^2} \left(\int_{\mathcal{X}}  ( f(x)- g(x;[\Theta^{*}])  )  \tilde C^{\ast, c}
\sigma'(\tilde Z^{\ast, c}(x) )     \sigma ( \tilde W^{1,*,w}\cdot  x) \pi(dx)\right)^{2}\mu_{W^2}(du) \mu_{W^{1}}(dw)\mu_c(dc) \notag \\
&=  \int_{\mathcal{C}}\int_{\mathcal{W}^{1}}  \left(\int_{\mathcal{X}}  ( f(x)- g(x;[\Theta^{*}])  )  \tilde C^{\ast, c}
\sigma'(\tilde Z^{\ast, c}(x) )    \sigma ( \tilde W^{1,*,w}\cdot  x)  \pi(dx)\right)^{2}  \tilde{\mu}_{W^{1}}(w) \tilde{\mu}_c(c) dw dc.
\label{EqnJ1}
\end{align}

Equation (\ref{EqnJ1}) and Assumption \ref{A:SupportInitial} imply that
\begin{eqnarray}
\int_{\mathcal{X}}  ( f(x)- g(x;[\Theta^{*}])  )  \tilde C^{\ast, c} \sigma'(\tilde Z^{\ast, c}(x) )    \sigma ( \tilde W^{1,*,w}\cdot  x)  \pi(dx)  \tilde{\mu}_{W^{1}}(w) \tilde{\mu}_c(c) = 0,
\label{EqualityJ10}
\end{eqnarray}
for almost every $w \in \mathcal{W}^{1}$ and $c \in \mathcal{C}$. Recalling now that
\begin{eqnarray}
Z^{\ast, c}(x) &=&  \int_{\mathcal{W}^{1}} \int_{\mathcal{W}^2}   \tilde  W^{2, \ast, c,w,u }    \sigma ( \tilde W^{1,*,w}\cdot  x)  \mu_{W^2}(du) \mu_{W^{1}}(dw), \notag
\end{eqnarray}
we obtain
\begin{eqnarray}
\int_{\mathcal{C}} \sup_{x \in \mathcal{X}} | Z^{\ast, c}(x) | \mu_c(dc) &\leq& \int_{\mathcal{C}}  \int_{\mathcal{W}^{1}} \int_{\mathcal{W}^2} | \tilde  W^{2, \ast, c,w,u } |   \mu_{W^2}(du) \mu_{W^{1}}(dw) \mu_c(dc) \notag \\
&\leq& \int_{\mathbb{R}^d}  | u | \tilde{\eta}^{\ast}_{\mathcal{W}^{2}}(u) du < \infty,\notag
\end{eqnarray}
where $ \tilde{\eta}^{\ast}_{W^{2}}(u) $ denotes the marginal density of the limiting measure $\eta^{*}$ with respect to $\tilde{W}^{2}$. Analogously we shall define the marginal densities  $ \tilde{\eta}^{\ast}_{c}(c) $ and  $ \tilde{\eta}^{\ast}_{W^{1}}(w) $ with respect to the variables $\mathcal{C}$ and $\tilde{W}^{1}$ respectively.  Similar to before, $\tilde{\eta}^{\ast}(c, w, u)$ is the density of the measure $\eta^{\ast}(dc, dw, du)$.

We must therefore have that $\displaystyle \sup_{x \in \mathcal{X}} | Z^{\ast, c} (x) | \tilde{\mu}_c(c) < \infty$ for almost every $c \in \mathcal{C}$. In addition, since $\eta^{\ast}$ has a density, we have that there exists a set $B = \{ c \in \mathbb{R}: | c - c_0 | < \delta \}$ with $ c_0 \neq 0$ and $0 < \delta < | c_0 |$ such that
\begin{eqnarray}
0 &<& \int_{\mathbb{R}} \mathbf{1}_{ \{ c \in B \} } \tilde{\eta}^{\ast}_{c}(c) dc \notag \\
&=& \int_{\mathcal{C}} \mathbf{1}_{\{ C^{\ast, c} \in B \}} \tilde{\mu}_c(c) dc \notag \\
&=& \int_{\mathcal{C}} \mathbf{1}_{\{ C^{\ast, c} \in B, \displaystyle  \sup_{x \in \mathcal{X}} | Z^{\ast, c}(x) | \tilde{\mu}_c(c) < \infty \} } \tilde{\mu}_c(c) dc.\notag
\end{eqnarray}

Therefore, there must be a set $K \subset \mathbb{R}$ with Lesbegue measure $\lambda(K) > 0$ such that, for almost every $c \in K$,
\begin{eqnarray}
\mathbf{1}_{ \{ C^{\ast, c} \in B, \displaystyle \sup_{x \in \mathcal{X}} | Z^{\ast, c}(x) | \tilde{\mu}_c(c) < \infty \} } \tilde{\mu}_c(c) > 0,
\label{InEqualityJ10}
\end{eqnarray}
which in turn implies that $C^{\ast, c} \neq 0$ and $\displaystyle \sup_{x \in \mathcal{X}} | Z^{\ast, c}(x) | < \infty$ for almost every $c \in K$. Due to equations (\ref{InEqualityJ10}) and (\ref{EqualityJ10}), for almost every $(c,w) \in K \times \mathcal{W}^{1}$,
\begin{eqnarray}
 C^{\ast, c}  &\neq& 0, \notag \\
 \sup_{x \in \mathcal{X}} | Z^{\ast, c}(x) | &<& \infty, \notag \\
0 &=& \int_{\mathcal{X}}  ( f(x)- g(x;[\Theta^{*}])  )  \tilde C^{\ast, c} \sigma'(\tilde Z^{\ast, c}(x) )    \sigma ( \tilde W^{1,*,w}\cdot  x)  \pi(dx)  \tilde{\mu}_{W^{1}}(w) \tilde{\mu}_c(c).
\label{JthreeConditions}
\end{eqnarray}

Thus, there exists a $c_0 \in K$ such that
\begin{eqnarray}
\int_{\mathcal{X}}  ( f(x)- g(x;[\Theta^{*}])  )   C^{\ast, c_0}   \sigma'(\tilde Z^{\ast, c_0}(x) )    \sigma ( \tilde W^{1,*,w}\cdot  x)  \pi(dx)  \tilde{\mu}_{W^{1}}(w) = 0,\notag
\end{eqnarray}
for almost every $w \in \mathcal{W}^{1}$. This of course implies that, for almost every $w \in \mathcal{W}^{1}$,
\begin{eqnarray}
\int_{\mathcal{X}}  ( f(x)- g(x;[\Theta^{*}])  )     \sigma'(\tilde Z^{\ast, c_0}(x) )    \sigma ( \tilde W^{1,*,w}\cdot  x)  \pi(dx)  \tilde{\mu}_{W^{1}}(w) = 0.\notag
\end{eqnarray}

Consequently, we have that
\begin{align}
 \int_{\mathcal{W}^1} \int_{\mathcal{X}}  ( f(x)- g(x;[\Theta^{*}])  )   \sigma'(\tilde Z^{\ast, c_0}(x) )    \sigma ( \tilde W^{1,*,w}\cdot  x)  \pi(dx)  \tilde{\mu}_{W^{1}}(w) dw&=0. \notag
\end{align}

Therefore, since $\eta^{\ast}_{W^{1}}$ has a density, there exists a compact set $A \subset \mathbb{R}^d$ such that, for every $w \in A$,
\begin{align}
\Gamma(w) &= \int_{\mathcal{X}}  ( f(x)- g(x;[\Theta^{*}])  )   \sigma'(\tilde Z^{\ast, c_0}(x) )    \sigma (  w \cdot  x)  \pi(dx) = 0.\label{Eq:GammaFcn}
\end{align}

Since $\sigma(\cdot)$ is analytic, $\Gamma(w)$ is analytic (see Lemma \ref{L:AnalyticityProof}). Therefore, by the identity theorem for real-analytic functions, $\Gamma(w) = 0$ for every $w \in \mathbb{R}^d$. It then follows, since $\sigma(\cdot)$ is discriminatory,
\begin{eqnarray}
( f(x)- g(x;[\Theta^{*}])  )   \sigma'(\tilde Z^{\ast, c_0}(x) )  = 0,\notag
\end{eqnarray}
for every $x \in \mathcal{X}$. Finally, since $ \sigma'(\tilde Z^{\ast, c_0}(x) )  > 0$ due to equation (\ref{JthreeConditions}),
\begin{eqnarray}
 f(x)- g(x;[\Theta^{*}])   = 0,\notag
\end{eqnarray}
for every $x \in \mathcal{X}$, concluding the proof of the theorem.
\end{proof}

We conclude with the the result on analyticity of $\Gamma(w)$ defined in (\ref{Eq:GammaFcn}).
\begin{lemma} \label{L:AnalyticityProof}
Assume $\mathcal{X}$ is a compact set and the activation unit $\sigma(\cdot)$ is real analytic and bounded. Then,
\begin{align}
\Gamma(w) &= \int_{\mathcal{X}}  ( f(x)-g(x;[\Theta^{*}])  )   \sigma'(\tilde Z^{\ast, c_0}(x) )    \sigma (  w \cdot  x)  \pi(dx)\notag
\end{align}
defined in (\ref{Eq:GammaFcn}) is a real analytic function as well.
\end{lemma}
\begin{proof}[Proof of Lemma \ref{L:AnalyticityProof}]
The partial derivative of $\sigma (w \cdot x)$ is
\begin{eqnarray}
 \frac{\partial^{ | \mu |}}{ \partial w^{\mu}} \sigma ( w \cdot x)  = \sigma^{|\mu|} ( w \cdot x) \prod_{i=1}^d x^{\mu_i}.\notag
\end{eqnarray}

Before we prove analyticity of $\Gamma(w)$, we show that $g(x)$ is finite. Recall that
\begin{eqnarray}
g(x;[\Theta^{*}]) = \int_{\mathcal{C}} C^{\ast, c} \sigma(Z^{\ast, c}(x) ) \tilde{\mu}_c(c) dc.\notag
\end{eqnarray}
Therefore, due to $\sigma(\cdot)$ being bounded and $\eta^{\ast}$ having marginals with finite first moments,
\begin{eqnarray}
| g(x) | &\leq& K_0 \int_{\mathcal{C}} |C^{\ast, c} |\tilde{\mu}_c(c) dc \notag \\
&\leq& K_0 \int_{\mathcal{C}} | c| \tilde{\eta}^{\ast}_{c}(c) dc \notag \\
&\leq& K.\notag
\end{eqnarray}

Next,
\begin{eqnarray}
 \frac{\partial^{ | \mu |} \Gamma}{ \partial w^{\mu}}(w) =  \int_{\mathcal{X}}  ( f(x)- g(x;[\Theta^{*}])  )   \sigma'(\tilde Z^{\ast, c_0}(x) )    \sigma^{|\mu|} ( w \cdot x) \prod_{i=1}^d x^{\mu_i}  \pi(dx).\notag
\end{eqnarray}

Due to the compactness of $\mathcal{X}$, $\displaystyle \sup_{x \in \mathcal{X}} | \sigma'(\tilde Z^{\ast, c_0}(x) ) | < \infty$, and the fact that $\sigma$ is a real analytic function, Proposition 2.2.10 of  \cite{KrantzParks2002} shows there exist constants $C_1, C_2 > 0$ such that
\begin{eqnarray}
\bigg{|} \frac{\partial^{ | \mu |} \Gamma}{ \partial w^{\mu}}(w)  \bigg{|} &\leq& C_1 \frac{\mu !}{R_1^{|\mu|}}  \prod_{i=1}^d C_2^{\mu_i}  \notag \\
&=& C_1 \frac{\mu !}{R_1^{|\mu|}}   C_2^{| \mu | } \notag \\
&=& C_1 \frac{\mu !}{R_1^{|\mu|}}   ( R_2^{-1})^{| \mu | } \notag \\
&=& C_1 \frac{\mu !}{(R_1 R_2)^{|\mu|}} \notag \\
&=& C_1 \frac{\mu !}{R^{|\mu|}},\notag
\end{eqnarray}
where $R_2 = \frac{1}{C_2}$ and $R = R_1 R_2$. It follows then again from Proposition 2.2.10 of  \cite{KrantzParks2002} that $\Gamma(w)$ is a real analytic function. This concludes the proof of the lemma.
\end{proof}

\section{Discussion on the limiting results and extensions to multi-layer networks with greater depth}\label{S:Consequences}

In Subsection \ref{SS:Disucssion},  we discuss some of the implications of our theoretical convergence results and presents related numerical results. In Section \ref{SS:Extension}, we show that the procedure can be extended to treat deep neural networks with more than two hidden layers.  General challenges in the study of multi-layer neural networks are explored Subsection \ref{ChallengesIntro}.

\subsection{Discussion on the limiting results}\label{SS:Disucssion}

It is instructive to notice that the results of this paper recover the results of \cite{NeuralNetworkLLN} (see also \cite{Montanari,Rotskoff_VandenEijnden2018}) if we restrict attention to the one-layer case.  Indeed, let us set $N_{2}=1$, $N_1=N$, $C^{i}=1$ and $H^{2,i}=Z^{2,i}$ in (\ref{Eq:2DNN})-(\ref{Eq:DeepNN}), and we get the single-layer neural network
\begin{eqnarray}
g_{\theta}^{N}(x) =  \frac{1}{N} \sum_{j=1}^{N}W^{2,j}\sigma\left(W^{1,j}\cdot x\right),\notag
\end{eqnarray}
with the corresponding empirical measure of the parameters becoming
\begin{eqnarray}
\gamma_t^{N} = \tilde \gamma_{\floor*{N t} }^{N}, \text{ where } \tilde \gamma_k^{N} = \frac{1}{N} \sum_{j=1}^{N} \delta_{ W_k^{1,j}, W^{2,j}_{k}}.\label{Eq:EmpiricalMeasureOneLayer}
\end{eqnarray}

In that case notice that we can simply write
\begin{eqnarray}
g_{\theta_{\floor*{N t} }}^{N}(x)  =   \la w^2 \sigma(w^{1} \cdot x), \gamma^{N}_t \ra .\nonumber
\end{eqnarray}

Then, it is relatively straightforward to notice that the result of Lemma \ref{LemmaFirstLayermaintext} boils down to the one layer convergence results of \cite{NeuralNetworkLLN}, see also \cite{Montanari,Rotskoff_VandenEijnden2018}. Namely, if we write $\gamma_t$ for the limit in probability of $\gamma^{N}_t$ we get that
\begin{align}
\lim_{N \rightarrow \infty} g_{\theta_{\floor*{N t} }}^{N}(x)=\la w^2 \sigma(w^{1} \cdot x), \gamma_t \ra.\label{Eq:CompressedNNoutput1L}
\end{align}

It is useful to compare the limits of the neural network output in the one layer and two layer cases, (\ref{Eq:CompressedNNoutput1L}) and (\ref{Eq:CompressedNNoutput}) respectively.

Some general remarks of interest follow.
\begin{itemize}
\item{It is clear that the two layer case (and more generally the multi-layer case) is more complex than the single-layer case, which provides some intuition for the increased complexity of deep neural networks when compared to shallow neural networks. For instance, in the single layer case the neural network output is given explicitly by (\ref{Eq:CompressedNNoutput1L}). On the other hand, in the multi-layer case, the neural network's output is the solution to an integral equation as given by (\ref{Eq:LLN_DNNTheoremIntro})-(\ref{Eq:CompressedNNoutput}).}
\item{In contrast to the single-layer case, in the multi-layer case the asymptotic weight distribution is layer dependent and weights do not necessarily become independent in the large hidden units limit. In fact, as our result demonstrates, see (\ref{Eq:LLN_DNNTheoremIntro0})-(\ref{Eq:LLN_DNNTheoremIntro}), the neural network output in the large hidden units limit and large SGD-iterates does not depend on the individual weights, but on paths of weights that connect the input layer to the output layer. One could regard these paths, i.e. the solution to the random ODE's in (\ref{Eq:LLN_DNNTheoremIntro}), as typical representative paths of the weights in the limit connecting the different layers of the ``limiting" neural network.}
\item{In the multi-layer case, units at a given layer receive an aggregate signal from units of the previous and/or of the next layer.}
\item{The law of large numbers for a single-layer network indicates that the network will converge in probability to a deterministic limit. That is, after a certain point, adding more hidden units will not increase the accuracy. Our main result Theorem \ref{Theorem1}
suggests that the same conclusion is true for multi-layer neural networks as well.
    }
\item{Theorem \ref{T:GlobalConvergenceTheorem} combined with Theorem \ref{Theorem1} characterizes the limiting behavior of the objective function $L^{N_{1},N_{2}}(\theta)$ from (\ref{LossFunction}). Section \ref{S:GlobalConvergence} shows that, under the proper assumptions, the limit objective function decreases in the gradient direction of the paths governing the limiting behavior of the weights. The limit ODEs seek to minimize the limit objective function and the global minimum is expected to be recovered.}
\end{itemize}

One practical consequence of our analysis is that the parametrization of the learning rates, see (\ref{Eq:LearningRate}), indicates that one should use larger learning rates for the weights that connect the different layers (e.g., $W^{2,i,j}$) as compared to the weights for the input or output layers (e.g., $C^{i}$ and $W^{1,j}$). Notice that this is also the case for the three layer case outlined in Subsection \ref{SS:Extension} below.

As will be explained in Section \ref{SS:Extension}, the law of large numbers can be extended to multi-layer neural networks with an arbitrary number of layers. The law of large numbers will only hold under a certain choice of the learning rates. The learning rates need to be scaled with the number of hidden units in each layer. For a multi-layer network with $L$ layers, the learning rates are
\begin{eqnarray}
\alpha_C =\frac{N_L}{N_1},\quad \alpha_{W,1} =1, \quad \alpha_{W,2} = N_2, \quad \alpha_{W,L} =\frac{N_{L-1} N_L}{N_1},\quad \alpha_{W,\ell} = \frac{N_{\ell-1} N_{\ell}}{N_1},
\end{eqnarray}
where $N_{\ell}$ is the number of hidden units in the $\ell$-th layer.

If the learning rates are constant in the number of hidden units $N_1, \ldots, N_L$, it turns out that the network will not train as $N_1, \ldots, N_L \rightarrow \infty$ (i.e., in the limit, the network parameters will remain at their initial conditions).

The necessity of scaling the learning rates in the asymptotic regime of large numbers of hidden units (i.e., wide layers) is one of the interesting products of the mean-field limit analysis. A numerical example is presented in Figure \ref{CIFAR10scaledLR} below. A deep neural network is trained to classify images for the CIFAR10 dataset \cite{Krizhevsky}. The CIFAR10 dataset contains $60,000$ color images in $10$ classes (airplane, automobile, bird, cat, deer, dog, frog, horse, ship, and truck). The dataset is divided into $50,000$ training images and $10,000$ test images. Each image has $32 \times 32 \times 3$ pixels. The goal is to train a neural network to correctly classify each image based solely on the image pixels as an input. The neural network we use has the mean-field normalizations $\frac{1}{N_{\ell}}$ in each layer $\ell$. There are $8$ convolution layers which each have $64$ channels. This is followed by two fully-connected layers which each have $500$ units. We first train the neural network using the scaled learning rates. Then, we also train the neural network with the standard stochastic gradient descent algorithm (no scaling of the learning rates). Using the scaled learning rates, we achieve a high test accuracy. However, without the scalings, the neural network does not train (i.e., it remains at a very low accuracy).
\begin{figure}[ht!]
\begin{center}
\includegraphics[width=.47\textwidth]{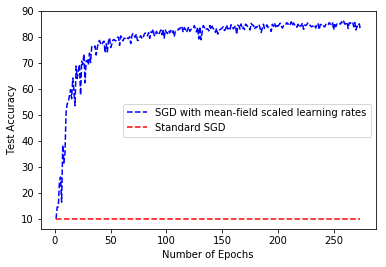}
\end{center}
\caption{Performance of deep neural network on CIFAR10 dataset with and without scaled learning rates.}
\label{CIFAR10scaledLR}
\end{figure}

\subsection{Extension to multi-layer neural networks with more layers}\label{SS:Extension}

The procedure developed in this paper naturally extends to multi-layer neural networks with more layers than two layers. For brevity, let us present the result in the case of three layers. The situation for more layers is the same, albeit with more complicated algebra. A multi-layer neural network with three layers takes the form
\begin{eqnarray}
g_{\theta}^{N_1, N_2, N_3}(x) = \frac{1}{N_{3}} \sum_{i=1}^{N_{3}} C^i \sigma\left( \frac{1}{N_{2}} \sum_{j=1}^{N_{2}}W^{3,i,j}\sigma\left(\frac{1}{N_{1}} \sum_{\nu=1}^{N_{1}}W^{2,j,\nu}\sigma\left(W^{1,\nu}\cdot x\right)\right)\right),\label{Eq:3DNN}
\end{eqnarray}
which can be also written as
\begin{eqnarray}
H^{1,\nu}(x) &=& \sigma( W^{1,\nu}\cdot x), \phantom{.....} \nu =1, \ldots, N_1, \notag \\
Z^{2,j}(x) &=& \frac{1}{N_1} \sum_{\nu=1}^{N_1} W^{2,j,\nu} H^{1,\nu}(x), \phantom{.....} j=1, \ldots, N_2\notag \\
H^{2,j}(x) &=& \sigma \bigg{(} Z^{2,j}(x) \bigg{)}, \phantom{.....}    \notag \\
Z^{3,i}(x) &=& \frac{1}{N_2} \sum_{j=1}^{N_2} W^{3,i,j} H^{2,j}(x), \phantom{.....} i =1, \ldots, N_3 \notag \\
H^{3,i}(x) &=& \sigma \bigg{(} Z^{3,i}(x) \bigg{)}, \phantom{.....}    \notag \\
g^{N_1, N_2,N_3}_{\theta}(x) &=& \frac{1}{N_3} \sum_{i=1}^{N_3} C^i H^{3,i}(x). \label{Eq:3DeepNN}
\end{eqnarray}
where $C^i, W^{2,j,\nu}, W^{3,i,j} \in \mathbb{R}$ and $x, W^{1,\nu} \in \mathbb{R}^{d}$.  The neural network model has parameters
\begin{eqnarray}
\theta = (C^1, \ldots, C^{N_3}, W^{2,1,1} \ldots, W^{2,N_2,N_1}, W^{3,1,1} \ldots, W^{3,N_3,N_2}, W^{1,1},\ldots W^{1,N_1}), \notag
\end{eqnarray}
 which must be estimated from data. The number of hidden units in the first layer is $N_1$, the number of hidden units in the second layer is $N_2$, and the number of hidden units in the third layer is $N_3$. Naturally, the loss function now becomes
\begin{eqnarray}
L^{N_1, N_2, N_3}(\theta) =\frac{1}{2} \mathbb{E}_{Y,X} \bigg{[} ( Y - g^{N_1, N_2,N_3}_{\theta}(X)  )^2 \bigg{]},\notag
\end{eqnarray}
where the data $(X,Y) \sim \pi(dx,dy)$.

The stochastic gradient descent (SGD) algorithm for estimating the parameters $\theta$ is, for $k \in\mathbb{N}$, $\nu =1, \ldots, N_1$, $i =1, \ldots, N_3$ and $j=1, \ldots, N_2$ is
\begin{align}
&C^i_{k+1} = C^i_{k} + \frac{\alpha_C^{N_1, N_2,N_3} }{N_3} \big{(} y_k - g^{N_1, N_2, N_3}_{\theta_k}(x_k)  \big{)} H_k^{3,i}(x_k), \notag \\
&W^{1,\nu}_{k+1} = W^{1,\nu}_{k} +  \frac{ \alpha_{W,1}^{N_1, N_2, N_3} }{ N_1} \big{(} y_k - g^{N_1, N_2, N_3}_{\theta_k}(x_k)  \big{)} \left( \frac{1}{N_3} \sum_{i=1}^{N_3} C_k^i \sigma'(Z_k^{3, i}(x_k) ) \left(\frac{1}{N_2}\sum_{j=1}^{N_2} W^{3,i,j}\sigma'(Z^{2,j}(x_{k})) W^{2,j,\nu}_k \right)\right)\times\notag\\
&\qquad\qquad\times  \sigma'(W^{1,\nu}_k \cdot x_k) x_k, \notag \\
&W^{3,i,j}_{k+1} = W^{3,i,j}_{k}  + \frac{\alpha_{W,3}^{N_1, N_2, N_3} }{N_2 N_3}  \big{(} y_k - g^{N_1, N_2, N_3}_{\theta_k}(x_k) \big{)} C_k^i \sigma'(Z_k^{3,i}(x_k) )  H^{2,j}_k(x_k), \notag \\
&W^{2,j,\nu}_{k+1} = W^{2,j,\nu}_{k}  + \frac{\alpha_{W,2}^{N_1, N_2, N_3} }{N_1 N_2 }  \big{(} y_k - g^{N_1, N_2, N_3}_{\theta_k}(x_k) \big{)}  \frac{1}{N_3} \sum_{i=1}^{N_3} C_k^i \sigma'(Z_k^{3,i}(x_k) ) W^{3,i,j}_{k}  \sigma'(Z_k^{2,j}(x_k) )  H^{1,\nu}_k(x_k), \notag \\
\end{align}
where
\begin{align}
&H_k^{1,\nu}(x_k) = \sigma( W^{1,\nu}_k\cdot x_k), \phantom{.....} \nu =1, \ldots, N_1, \notag \\
&Z_k^{2,j}(x_k) = \frac{1}{N_1} \sum_{\nu=1}^{N_1} W_k^{2,j,\nu} H_k^{1,\nu}(x_k), \notag \\
&H_k^{2,j}(x_k) = \sigma \bigg{(} Z_k^{2,j}(x_k) \bigg{)}, \phantom{.....}    \notag \\
&Z_k^{3,i}(x_k) = \frac{1}{N_2} \sum_{j=1}^{N_2} W_k^{3,i,j} H_k^{2,j}(x_k), \phantom{.....} \notag \\
&H_k^{3,i}(x_k) = \sigma \bigg{(} Z_k^{3,i}(x_k) \bigg{)}, \phantom{.....}    \notag \\
&g^{N_1, N_2,N_3}_{\theta_k}(x_k) = \frac{1}{N_3} \sum_{i=1}^{N_3} C^i_k H_k^{3,i}(x_k).
\label{SystemDeep_3}
\end{align}
where $\alpha_C^{N_1, N_2, N_3}$, $\alpha_{W,1}^{N_1, N_2, N_3}$, $\alpha_{W,2}^{N_1, N_2, N_3}$ and $\alpha_{W,3}^{N_1, N_2, N_3}$ are the learning rates.  The parameters at step $k$ are
\[
\theta^{N_1,N_2,N_3}_k = (C^1_k, \ldots, C^{N_3}_k, W^{2,1,1}_k \ldots, W^{2,N_2,N_1}_k, W^{3,1,1}_k \ldots, W^{3,N_3,N_2}_k, W^{1,1}_k,\ldots W^{1,N_1}_k).
\]
 $(x_k, y_k)$ are samples of the random variables $(X,Y)$. We assume a condition analogous to Assumption \ref{A:Assumption1}.

Let us now choose the learning rates to be
\[
\alpha_C^{N_1, N_2, N_3}=\frac{N_3}{N_1},\quad \alpha_{W,1}^{N_1, N_2, N_3}=1, \quad \alpha_{W,3}^{N_1, N_2, N_3}=\frac{N_2 N_3}{N_1},\quad \alpha_{W,2}^{N_1, N_2, N_3}=N_2
\]

Similar to before, define the empirical measure
\begin{eqnarray}
\tilde \gamma_k^{N_1, N_2, N_3} &=& \frac{1}{N_1} \sum_{\nu=1}^{N_1} \delta_{ W^{1,\nu}_k, W^{2,1,\nu}_k,\ldots, W^{2,N_2,\nu}_k, W^{3,1,1}_k \ldots, W^{3,N_3,N_2}_k, C^1_k, \ldots, C^{N_3}_k }.\notag
\end{eqnarray}

The time-scaled empirical measure is
\begin{eqnarray}
\gamma_t^{N_1, N_2,N_3} \vcentcolon= \tilde \gamma_{\floor*{N_{1} t} }^{N_1, N_2, N_3},\nonumber
\end{eqnarray}
and the corresponding time-scaled neural network output is $g_{t}^{N_1, N_2, N_3}(x) = g_{\theta_{ \floor*{N_{1} t} }}^{N_1, N_2,N_3}(x)$.

Following the same procedure as for the two layer case one expects to get the following limit that describes $g_{t}(x)$ (the iterated limit of $g_{t}^{N_1, N_2, N_3}(x)$ as first $N_1\rightarrow\infty$, then $N_2\rightarrow\infty$, and then $N_3\rightarrow\infty$):
\begin{align}
d \tilde C_t^{c} &=  \int_{\mathcal{X} \times \mathcal{Y} } \big{(} y -  g_t(x)  \big{)} \tilde H_t^{3,c}(x) \pi(dx,dy)  dt, \phantom{....} 
\tilde C_0^{c} = c, \notag \\
 d \tilde W_t^{1,w} &=  \int_{\mathcal{X} \times \mathcal{Y} } \big{(} y - g_t(x)  \big{)} V_t^{w}(x) \sigma'( W_t^{1,w} \cdot x) x \pi(dx,dy)   dt, \phantom{....}
 \tilde W_0^{1,w} = w, \notag \\
d \tilde W^{3,c,v}_{t} &=  \int_{\mathcal{X} \times \mathcal{Y} } \big{(} y - g_t(x) \big{)} \tilde C_t^{c} \sigma'(\tilde Z_t^{3,c}(x) )  \tilde H_t^{2,v}(x) \pi(dx,dy)  dt,  \phantom{......} \tilde W^{3,c,v}_0 = v,  \notag \\
d \tilde W^{2,w,u,v}_{t} &=  \int_{\mathcal{X} \times \mathcal{Y} } \big{(} y - g_t(x) \big{)} L_t^{v}(x) \sigma'(\tilde Z_t^{2,v}(x) ) \tilde H_t^{1,w}(x) \pi(dx,dy)  dt,  \phantom{......} \tilde W^{2,w,u,v}_0 = u,  \notag \\
\tilde H_t^{1,w}(x) &= \sigma ( \tilde W_t^{1,w}\cdot  x), \notag \\
\tilde Z_t^{2,v}(x) &=  \int_{\mathcal{W}^{1}} \int_{\mathcal{W}^2}   \tilde  W^{2,w,u,v}_t  \tilde H^{1,w}_t(x)  \mu_{W^2}(du) \mu_{W^{1}}(dw) , \notag \\
\tilde H_t^{2,v}(x) &= \sigma( \tilde Z_t^{2,v}(x) ), \notag \\
\tilde Z_t^{3,c}(x) &=  \int_{\mathcal{W}^3}   \tilde  W^{3,c,v}_t  \tilde H^{2,v}_t(x)  \mu_{W^3}(dv) , \notag \\
\tilde H_t^{3,c}(x) &= \sigma( \tilde Z_t^{3,c}(x) ), \notag \\
V_t^{w}(x) &= \int_{\mathcal{C}}  \tilde C_t^c  \sigma'(\tilde Z_t^{3,c}(x) ) \left( \int_{\mathcal{W}^3}  \tilde W^{3,c,v}_t  \sigma'(\tilde Z_t^{2,v}(x) )\left(\int_{\mathcal{W}^2}  \tilde W^{2,w,u,v}_t \mu_{W^2}(du) \right)\mu_{W^3}(dv)\right) \mu_c(dc)\notag\\
L_t^{v}(x) &= \int_{\mathcal{C}}  \tilde C_t^c  \sigma'(\tilde Z_t^{3,c}(x) )   \tilde W^{3,c,v}_t \mu_c(dc)\notag\\
g_t(x) &= \int_{\mathcal{C}}  \tilde C_t^{c} \tilde H_t^{3,c}(x) \mu_c(dc),  \label{Eq:LLN_DNNTheoremIntro3}
\end{align}

In other words, one expects to be able to write that the neural network's output is
\[
g_t(x)=\int_{\mathcal{C}}  \tilde C_t^{c} \left(\sigma\left( \int_{\mathcal{W}^3}   \tilde  W^{3,c,v}_t  \sigma\left( \int_{\mathcal{W}^1} \int_{\mathcal{W}^2}   \tilde  W^{2,w,u,v}_t  \tilde \sigma ( \tilde W_t^{1,w} \cdot x)  \mu_{W^2}(du) \mu_{W^{1}}(dw) \right)  \mu_{W^3}(dv) \right)\right) \mu_c(dc).
\]

A computation analogous to the one in Section \ref{S:GlobalConvergence} shows that, as expected, the limit ODEs in (\ref{Eq:LLN_DNNTheoremIntro3}) seek to minimize the corresponding limit objective function as $t \rightarrow \infty$. We leave the rigorous proof for the form of $g_{t}(x)$ in the three layer neural network case to the interested reader.

\subsection{Challenges in the analysis of multi-layer neural networks} \label{ChallengesIntro}

Challenges arise in the study of the asymptotics of multi-layer neural networks. \cite{NeuralNetworkLLN}, \cite{Montanari}, and \cite{Rotskoff_VandenEijnden2018} leveraged traditional approaches from the mean-field and interacting particle system literature to characterize the asymptotics of single-layer neural networks. However, it turns out that the traditional mean-field approach cannot be used for multi-layer neural networks.

A standard approach for analyzing (\ref{SystemDeep}) as $N_1, N_2 \rightarrow \infty$ would be to construct an empirical measure $\rho_k^{N_1, N_2}$ of the parameters $\theta_k$ at training step $k$, as for example $\tilde \gamma_k^{N_1,N_2}$. Then, we could study the behavior of $\rho_k^{N_1, N_2}$ as $N_1, N_2 \rightarrow \infty$. This empirical measure $\rho_k^{N_1, N_2}$ needs to be designed such that the dynamics of $\rho_k^{N_1, N_2}$ can be written in terms of $\rho_k^{N_1, N_2}$ itself and the data $(x_k, y_k)$ (plus martingale and remainder terms). That is, the dynamics of $\rho_k^{N_1, N_2}$ are \emph{closed}.

This is straightforward for single-layer neural networks (see \cite{NeuralNetworkLLN}), but it is challenging to do for multi-layer neural networks. In the case of single-layer networks, the empirical measure is simply given by (\ref{Eq:EmpiricalMeasureOneLayer}) and its analysis has been successfully carried on in \cite{NeuralNetworkLLN}. One is tempted to do the same thing for the multi-layer case, i.e. study the limit of the empirical measure defined in (\ref{Eq:EmpiricalMeasureMultilayer}). The problem that one faces with this formulation is that, in contrast to the single-layer case, the dimension of the space on which the empirical measure takes values increases with $N_2$. Therefore, the problem cannot be studied using such a empirical measure.

An alternative approach, which is also natural in this case, is to try and create ``nested measures", sometimes called multi-level measure valued processes in mathematical biology, see \cite{Dawson1997,DawsonHochberg1982,DawsonHochberg1991, DawsonHochberg1996, DawsonGartner1998, DawsonGreven1993, DawsonWu1996} and the review paper \cite{Dawson2018}.

Let us explain how to construct a nested empirical measure for (\ref{SystemDeep}) and why it will not work. In order to simplify notation, let $N_1 = N_2 = N$ and $\rho_k^{N_1, N_2} = \rho_k^N$ in the following example. Considering $N_1\neq N_2$ and taking subsequent limits does not alter the conclusions below.
\begin{itemize}
\item Let's first examine the parameter $W^{1,j}$ in the first layer. The $j$-th unit in the first layer (i.e., $H^{1,j} = \sigma( W^{1,j}\cdot x) $) is connected to all of the hidden units in the second layer (i.e., $H^{2,i}$) via the weights $W^{2,i,j}$. Therefore, there is a measure associated with each $W^{1,j}$ which must track $\{ W^{2,i,j}, Z^i, C^i \}_{i=1}^N$. $W^{2,i,j}$ and $Z^i$ are required for calculating the SGD update for $W^{1,j}$ (see \ref{SystemDeep}). This measure is $\nu^{N,j}_k = \frac{1}{N} \sum_{i=1}^N \delta_{ W^{2,i,j}_k, Z^i_k, C^i_k }$.
\item Let's next examine the parameter $C^i$ in the second layer. The $i$-th unit in the second layer is connected to all of the hidden units in the first layer via the weights $W^{2,i,j}$. Therefore, for each $C^i$, we must track $\{ W^{2,i,j},  W^{1,j} \}_{j=1}^N$ in order to calculate the SGD update for $C^i$ (see \ref{SystemDeep}). Furthermore, updating $C^i$ requires tracking $W^{1,j}$, and updating $W^{1,j}$ requires tracking $\nu^j$.  Therefore, updates to $C^i$ require the empirical measure $\mu^{N,i}_k = \frac{1}{N} \sum_{j=1}^N \delta_{ W^{2,i,j}_k , W^{1,j}_k, \nu^{N,j}_k } \in \mathcal{M}(  \mathbb{R} \times \mathbb{R}^d \times \mathcal{M} ( \mathbb{R}^3 ) )$.
\item Finally, the entire network at iteration $k$ is specified by the empirical measure
\begin{eqnarray}
\rho_k^N = \frac{1}{N} \sum_{i=1}^N \delta_{ C^i_k  , \mu^i_k } \in \mathcal{M} ( \mathbb{R} \times \mathcal{M}(  \mathbb{R} \times \mathbb{R}^d \times \mathcal{M}( \mathbb{R}^3 ) ) ),
\label{Defrhok}
\end{eqnarray}
\end{itemize}
where $\mathcal{M}(E)$ is the space of measures on the metric space $E$. Notice that the process (\ref{Defrhok}) involves \emph{nested measures} (sometimes called ``multi-level processes"). The process $\rho_k$ takes values in a \textit{space of nested measures} $\mathcal{M}( \mathcal{M} ( \mathcal{M} ( \cdots ) ) )$.

Careful inspection of $\rho_k^N$ identifies a crucial problem: its dynamics are not closed. The evolution of $\nu^{N,j}_k$ (the innermost measure in the nested measures)
cannot be written in terms of $\rho_k^N$. In particular, updating $\nu^{N,j}_k$ requires also updating $(W^{2,i,j}, Z^i_k, C^i_k)_{i=1}^N$. This would in fact require knowledge of $(W^{2,i,j}, Z^i_k, C^i_k, \mu^{N,i}_k)_{i=1}^N$, i.e. we would have to re-define $\nu^{N,j}_k$ as $\frac{1}{N} \sum_{i=1}^N \delta_{ W^{2,i,j}_k, Z^i_k, C^i_k, \eta^{N,i}_k }$ where $\eta^{N,i}_k = \frac{1}{N} \sum_{m=1}^N \delta_{ W^{2,i,m}_k , W^m_k }$. This leads to re-defining $\rho_k^N$ as taking values in $\mathcal{M} ( \mathbb{R} \times \mathcal{M}(  \mathbb{R} \times \mathbb{R}^d \times \mathcal{M} ( \mathbb{R}^3 \times \mathcal{M} ( \mathbb{R} \times \mathbb{R}^d ) ) ) )$, i.e. the space of $4$-nested measures (before it was $3$-nested measures). However, the closure problem remains, since the evolution of $\eta^{N,i}_k$ cannot be written in terms $\rho_k^N$. In fact, there does not seem to exist any finite number of nested measures for which the empirical measure $\rho_k^N$'s dynamics are closed and despite our best efforts we have not managed to find one. This closure problem then leads to non-trivial issues associated with establishing a well define limit.

For completeness, we also remark that a third alternative is to define the candidate empirical measure as an appropriately normalized double sum over both indices corresponding to the two hidden layers and then consider the limit of this measure as $N_1,N_2\rightarrow\infty$. However, this approach also suffers a closure problem, similar to the situation described above.

The discussion in this section highlights some of the challenges that must be addressed in the analysis of multi-layer neural networks. Such problems are not present in the analysis of single-layer neural networks. Therefore, a different approach is required for the asymptotic analysis of multi-layer neural networks and this paper is a first step in this direction.

The problems that we describe above led us to the approach used in this paper. In particular, the approach in Section \ref{ResultsIntro} first studies the limit of the empirical measure as the number of hidden units in the first layer grows to infinity. This is similar to \cite{NeuralNetworkLLN}. The limit is a solution to an evolution equation and it is the law of a system of random ODEs. We then make the crucial observation that one can characterize the resulting system in terms of the initialization for the stochastic gradient descent iterates. This means that we can reformulate the limiting system of the first layer into an equivalent system of random ODEs and then consider the limit of the second layer. This allows us to obtain the limit of the output of the neural network as the number of hidden units of all layers grow to infinity by studying the limit of the random ODE in  Theorem \ref{Theorem1}.

\section{Proof of Theorem \ref{Theorem1}-Characterization of the limit} \label{ProofSecondLayer}

In preparation for the proof of Theorem \ref{Theorem1}, we first re-express the result from Lemma \ref{LemmaFirstLayermaintext}.
\begin{corollary} \label{CorollaryFirstLayermaintext}
Consider the particle system:
\begin{eqnarray}
d C_t^i &=&  \int_{\mathcal{X} \times \mathcal{Y} } \big{(} y - g^{N_2}_t(x)  \big{)} H_t^{2,i}(x) \pi(dx,dy) dt, \phantom{.....} C^{i}_{0}=C^{i}_{\circ}, \phantom{....} i = 1, \ldots, N_2, \notag \\
 d W^{1}_t &=&  \int_{\mathcal{X} \times \mathcal{Y} } \big{(} y - g^{N_2}_t(x)  \big{)} \left( \frac{1}{N_2} \sum_{i=1}^{N_2} C_t^i \sigma'(Z_t^{i}(x) )  W^{2,i}_t \right) \sigma'(W^{1}_t \cdot x) x \pi(dx,dy)  dt, \phantom{.....} W^{1}_0 \sim \mu_{W^{1}}(dw), \notag \\
d W^{2,i}_{t} &=&   \int_{\mathcal{X} \times \mathcal{Y} } \big{(} y - g^{N_2}_t(x) \big{)} C_t^i \sigma'(Z_t^{i}(x) )  H_t^1(x) \pi(dx,dy)  dt, \phantom{.....} W^{2,i}_0 \sim \mu_{W^2}(du), \phantom{.....} i =1, \ldots, N_2, \notag \\
H_t^{1}(x) &=& \sigma ( W^{1}_t\cdot  x), \notag \\
Z_t^{i}(x) &=& \mathbb{E} \bigg{[} W^{2,i}_t  H^{1}_t(x) \bigg{|} C^1_{\circ}, \ldots, C^{N_2}_{\circ} \bigg{]}, \notag \\
H_t^{2,i}(x) &=& \sigma( Z_t^{i}(x) ), \notag \\
g^{N_2}_t(x) &=& \frac{1}{N_2} \sum_{i=1}^{N_2} C_t^i H_t^{2,i}(x).
\label{SystemLimitLayer1maintext}
\end{eqnarray}

Let $\nu_{t,(c_1, \ldots, c_{N_2})}$ be the conditional Law of $(W^{1}_t, W^{2,1}_t, \ldots W^{2,N_2}_t, C_t^{1}, \ldots, C_t^{N_2} )_{0 \leq t \leq T}$ given $(C^1_{0}, \ldots, C^{N_2}_{0}) = (c_1, \ldots, c_{N_2})$. Then, $\nu_{t,(C_{\circ}^1, \ldots, C_{\circ}^{N_2})}$ is the unique solution to the evolution equation (\ref{LimitSPDELLNmaintext})-(\ref{LimitSPDELLNmaintextb}).
\end{corollary}
\begin{proof}
See Appendix \ref{ProofFirstLayer}.
\end{proof}

Due to exchangeability properties and without loss of generality, we can transform (\ref{SystemLimitLayer1maintext}) into the equivalent particle system:

\begin{align}
d C_t^{C^{i}_{\circ}} &= \int_{\mathcal{X} \times \mathcal{Y} } \big{(} y - g^{N_2}_t(x)  \big{)} H_t^{2,C^{i}_{\circ}}(x) \pi(dx,dy)  dt, \quad C^{C^{i}_{\circ}}_{0}=C^{i}_{\circ}\phantom{......}
i = 1, \ldots, N_2, \notag \\
 d W_t^{1,W_0} &=  \int_{\mathcal{X} \times \mathcal{Y} } \big{(} y - g^{N_2}_t(x)  \big{)} V_t^{N_2, W_0} \sigma'( W_t^{1,W_0} \cdot x) x \pi(dx,dy)  dt,  \quad W_0^{1,W_0}=W_0\sim \mu_{W^{1}}(dw)\notag \\
d W^{2, C^{i}_{\circ}, W_0, W_0^{2, i}}_{t} &= \int_{\mathcal{X} \times \mathcal{Y} } \big{(} y - g^{N_2}_t(x) \big{)} C_t^{C^{i}_{\circ}} \sigma'(Z_t^{C^{i}_{\circ}}(x) )  H_t^{1,W_0}(x) \pi(dx,dy)  dt, \quad W^{2, C^{i}_{\circ}, W_0, W_0^{2, i}}_{0}=W_0^{2, i}\sim \mu_{W^{2}}(du), \notag \\
H_t^{1,W_0}(x) &= \sigma ( W_t^{1,W_0} \cdot x),  \notag \\
Z_t^{C^{i}_{\circ}}(x) &= \mathbb{E} \bigg{[}   W^{2,C^{i}_{\circ}, W_0, W_0^{2,i} }_t  H^{1, W_0}_t(x) \bigg{|} C^1_{\circ}, \ldots, C^{N_2}_{\circ} \bigg{]} , \phantom{.....} i = 1, \ldots, N_2,  \notag \\
H_t^{2,C^{i}_{\circ}}(x) &= \sigma( Z_t^{C^{i}_{\circ}}(x) ),  \phantom{.....} i =1, \ldots, N_2, \notag \\
g^{N_2}_t(x) &= \frac{1}{N_2} \sum_{i=1}^{N_2} C_t^{C^{i}_{\circ} } H_t^{2,C^{i}_{\circ}}(x), \notag \\
V_t^{N_2,W_0}(x) &= \frac{1}{N_2}  \sum_{i=1}^{N_2}  C_t^{C^{i}_{\circ}} \sigma'(Z_t^{C^{i}_{\circ} }(x) )  W^{2,C^{i}_{\circ}, W_0, W_0^{2,i} }_t.
 \label{TransformedSystemExchangeable0000}
\end{align}

Since for all $i=1,\cdots, N_2$, $C^{i}_{\circ}$ have probability density function as described in Assumption \ref{A:Assumption1}, we have that
\begin{eqnarray}
\mathbb{P} [ \{C^{i}_{\circ} \neq C_{\circ}^j \}_{i \neq j, (i,j) = 1,2, \ldots N_2 } ] = 1.\nonumber
\end{eqnarray}

This allows us to substitute the variable names
\begin{align*}
(\hat C_t^i,\hat W_t^{1}, \hat W_t^{2,i}, \hat Z_t^i, \hat H^1_t, \hat H_t^{2,i})& \textrm{ for } (C_t^{C^{i}_{\circ}}, W_t^{1,W_0}, W^{2, C^{i}_{\circ}, W_0, W_0^{2, i}}_{t}, Z_t^{C^{i}_{\circ}}, H^{1, W_0}_t, H_t^{2, C^{i}_{\circ}}),\notag
\end{align*}
for $i =1, \ldots, N_2$.

This produces the system:
\begin{eqnarray}
d \hat C_t^i &=& \int_{\mathcal{X} \times \mathcal{Y} } \big{(} y - g^{N_2}_t(x)  \big{)} \hat H_t^{2,i}(x) \pi(dx,dy)  dt, \phantom{.....} \hat C^{i}_{0} = C^{i}_{\circ}, \phantom{......}
i = 1, \ldots, N_2, \notag \\
 d \hat W^{1}_t &=&  \int_{\mathcal{X} \times \mathcal{Y} } \big{(} y - g^{N_2}_t(x)  \big{)} V_t^{N_2, W_0} \sigma'( \hat W^{1}_t \cdot x) x \pi(dx,dy)  dt, \phantom{.....} \hat W^{1}_0 = W_0\sim \mu_{W^{1}}(dw),  \notag \\
d \hat W_t^{2,i} &=&  \int_{\mathcal{X} \times \mathcal{Y} } \big{(} y - g^{N_2}_t(x) \big{)} \hat C_t^i \sigma'(\hat Z_t^i(x))  \hat H_t^{1}(x) \pi(dx,dy)  dt, \phantom{.....} \hat W_0^{2,i} = W_0^{2,i}\sim \mu_{W^{2}}(du), \notag \\
\hat H^1_t(x) &=& \sigma ( \hat W^{1}_t\cdot  x),  \notag \\
\hat Z_t^i(x) &=& \mathbb{E} \bigg{[}   \hat W_t^{2,i} \hat  H^1_t(x) \bigg{|} C^1_{\circ}, \ldots, C^{N_2}_{\circ} \bigg{]} , \phantom{.....} i = 1, \ldots, N_2,  \notag \\
\hat H_t^{2,i}(x) &=& \sigma( \hat Z_t^{i}(x) ),  \phantom{.....} i =1, \ldots, N_2, \notag \\
g^{N_2}_t(x) &=& \frac{1}{N_2} \sum_{i=1}^{N_2} \hat C_t^i \hat H_t^{2,i}(x), \notag \\
V_t^{N_2,W_0}(x) &=&  \frac{1}{N_2}  \sum_{i=1}^{N_2} \hat C_t^i  \sigma'(\hat Z_t^i (x) )  \hat W_t^{2,i} .
 \label{TransformedSystemExchangeable0000HAT}
\end{eqnarray}

The system (\ref{TransformedSystemExchangeable0000HAT}) is exactly the same system as (\ref{SystemLimitLayer1maintext}).
 Notice also that $\hat C_{t}^{i}$ in (\ref{TransformedSystemExchangeable0000HAT}) depends also on $\{C_{\circ}^{i}\}_{i=1}^{N_2}$ in a symmetric way via $g_{t}^{N_2}(x)$. Similarly $\hat W_{t}^{1}$ and $\hat W_{t}^{2,i}$ depend also on $\{C_{\circ}^{i}\}_{i=1}^{N_2}$ and on $\{W_{0}^{2,i}\}_{i=1}^{N_2}$ symmetrically via $g_{t}^{N_2}(x)$ and $V_t^{N_2,W_0}(x)$. This is also the situation for (\ref{SystemLimitLayer1maintext}).  Then independence and identical distribution of the initial conditions together with the aforementioned exchangeability property imply that  (\ref{SystemLimitLayer1maintext}) and (\ref{TransformedSystemExchangeable0000}) are equivalent.

\subsection{Limiting System} \label{LimitingSystem}

The goal is to prove that for any $t \in [0,1]$ and $x \in \mathcal{X}$,
\begin{eqnarray}
&\phantom{.}& \lim_{N_2 \rightarrow \infty}  g_{t}^{N_2}(x)   = g_t(x), \notag
\end{eqnarray}
in $L^{1}$, where $g_{t}^{N_2}(x)$ is from (\ref{TransformedSystemExchangeable0000}) and the limit $g_t(x)$ is given by
\[
g_t(x) =\int_{\mathcal{C}}  \tilde C_t^{c} \tilde H_t^{2,c}(x) \mu_c(dc),
\]
where,
\begin{eqnarray}
d \tilde C_t^{c} &=&  \int_{\mathcal{X} \times \mathcal{Y} } \big{(} y -  g_t(x)  \big{)} \tilde H_t^{2,c}(x) \pi(dx,dy)  dt, \phantom{....} 
\tilde C_0^{c} = c, \notag \\
 d \tilde W_t^{1,w} &=&  \int_{\mathcal{X} \times \mathcal{Y} } \big{(} y - g_t(x)  \big{)} V_t^{w} \sigma'(\tilde{W}_t^{1,w} \cdot x) x \pi(dx,dy)   dt, \phantom{....}
 \tilde W_0^{1,w} = w, \notag \\
d \tilde W^{2,c,w,u}_{t} &=&  \int_{\mathcal{X} \times \mathcal{Y} } \big{(} y - g_t(x) \big{)} \tilde C_t^{c} \sigma'(\tilde Z_t^{c}(x) )  \tilde H_t^{1,w}(x) \pi(dx,dy)  dt,  \phantom{......} \tilde W^{2,c,w,u}_0 = u,  \notag \\
\tilde H_t^{1,w}(x) &=& \sigma ( \tilde W_t^{1,w} \cdot x), \notag \\
\tilde Z_t^{c}(x) &=&  \int_{\mathcal{W}^{1}} \int_{\mathcal{W}^2}   \tilde  W^{2,c,w,u }_t  \tilde H^{1,w}_t(x)  \mu_{W^2}(du) \mu_{W^{1}}(dw) , \notag \\
\tilde H_t^{2,c}(x) &=& \sigma( \tilde Z_t^{c}(x) ), \notag \\
V_t^{w}(x) &=& \int_{\mathcal{C}}  \tilde C_t^c  \sigma'(\tilde Z_t^{c}(x) ) \left( \int_{\mathcal{W}^2}  \tilde W^{2,c,w,u}_t   \mu_{W^2}(du) \right) \mu_c(dc).
\label{Eq:LLN_DNN}
\end{eqnarray}

Before presenting the proof of this result, let us define a quantity that will be of central interest in the sequel. In particular, for $(c,w, u) \in \{(C_{\circ}^{i}, W^{1}_{0},W^{2,i}_{0}), i=1,\cdots N_{2}\}$ and $x\in \mathcal{X}$, let's define the error function
\begin{eqnarray}
E_t^{N_2}(c,w,u,x) & \vcentcolon =& ( C_t^c - \tilde C_t^c )^2 + \norm{ W_t^{1,w} - \tilde W_t^{1,w} }^2   + ( W_t^{2,c,w,u} - \tilde W_t^{2,c,w,u} )^2 \notag \\
& &\quad+   ( H_t^{2,c}(x) - \tilde H_t^{2,c}(x) )^2 + ( Z_t^c(x) - \tilde Z_t^{c}(x) )^2.\notag
\end{eqnarray}

Note that we certainly have,
\begin{eqnarray}
| C_0^c - \tilde C_0^c |^2 + \norm{W_0^{1,w} - \tilde W_0^{1,w} }^2 + | W_0^{2,c,w,u} - \tilde W_0^{2,c,w,u} |^2 + | Z_0^c(x) - \tilde  Z_0^c(x) |^2  = 0.\notag
\end{eqnarray}

We will show below that $\mathbb{E}\left[E_t^{N_2}(C^{i}_{\circ},W_{0},W_{0}^{2,i},x)\right]$ appropriately converges to zero as $N_2\rightarrow\infty$ which will then imply that $g_{t}^{N_2}(x)$ converges to $g_{t}(x)$ as indicated.

\subsection{A Priori Bounds}  \label{AprioriBoundsSecondLayer}
Let us first establish uniform bounds on the processes $C_t^c, W_t^{1,w}, W_t^{2,c,w,u},$ and $g_t^{N_2}(x)$ for the system (\ref{TransformedSystemExchangeable0000}). For any $t \in [0,T]$ and any $N_2 \in \mathbb{N}$,
\begin{eqnarray}
\frac{1}{N_2} \sum_{i=1}^{N_2} | C_t^{C^{i}_{\circ}} | \leq \frac{1}{N_2} \sum_{i=1}^{N_2}  | C_0^{C^{i}_{\circ}} | +  \int_0^t \int_{\mathcal{X} \times \mathcal{Y} } \big{|} y - g^{N_2}_s(x)  \big{|}  \frac{1}{N_2} \sum_{i=1}^{N_2} | H_s^{2,C^{i}_{\circ}}(x) | \pi(dx,dy)  ds.\nonumber
\end{eqnarray}

$| H_s^{2,C^{i}_{\circ}}(x) | < K$ since $\sigma(\cdot) \in C_b^{2}$. Then, using the fact that $\mathcal{X} \times \mathcal{Y}$ is compact,
\begin{eqnarray}
\frac{1}{N_2} \sum_{i=1}^{N_2} | C_t^{C^{i}_{\circ}} | \leq \frac{1}{N_2} \sum_{i=1}^{N_2} | C_0^{C^{i}_{\circ}}  | + K_1 t + K_2 \int_0^t  \frac{1}{N_2} \sum_{i=1}^{N_2} | C_s^{C^{i}_{\circ}} | ds.\nonumber
\end{eqnarray}

By Gronwall's inequality, we have that for any $N_2 \in \mathbb{N}$ and for any $t \in [0,T]$,
\begin{eqnarray}
\frac{1}{N_2} \sum_{i=1}^{N_2} | C_t^{C^{i}_{\circ}} | \leq K.\nonumber
\end{eqnarray}

Using the same approach, we can establish uniform bounds on the other processes $W_t^{1,w}$ and $W_t^{2,c,w,u}$. Therefore, for any $N_2 \in \mathbb{N}$, $t \in [0,1]$, $x \in \mathcal{X}$, and $(c,w, u) \in \{(C_{\circ}^{i}, W^{1}_{0},W^{2,i}_{0}), i=1,\cdots N_{2}\}$,
\begin{eqnarray}
\max\{| g_t^{N_2}(x) |,| C_t^c |,| W_t^{2,c,w,u} |,| V_t^{N_2,w}(x) |,| W_t^{1,w} |\} &\leq& K, \label{Eq:UniformBounds}
\end{eqnarray}

\subsection{Bound for $\mathbb{E}| V_t^{N_2,W_{0}}(x) - V_t^{W_{0}}(x) |^{2}$.}
We have:
\begin{align}
| V_t^{N_2, W_{0}}(x) - V_t^{W_{0}}(x) | &\leq \bigg{|}   \frac{1}{N_2}  \sum_{i=1}^{N_2}  C_t^{C^{i}_{\circ}} \sigma'(Z_t^{C^{i}_{\circ} }(x) )  W^{2,C^{i}_{\circ}, W_{0}, W_0^{2,i}}_t  - \int_{\mathcal{C}}  \tilde C_t^c  \sigma'(\tilde Z_t^{c}(x) ) \big{(} \int_{\mathcal{W}^2}  \tilde W^{2,c,W_{0},u}_t   \mu_{W^2}(du) \big{)} \mu_c(dc) \bigg{|} \notag \\
&=  \bigg{|} \frac{1}{N_2}  \sum_{i=1}^{N_2}  C_t^{C^{i}_{\circ}} \sigma'(Z_t^{C^{i}_{\circ} }(x) )  W^{2,C^{i}_{\circ}, W_{0}, W_0^{2,i}}_t -  \frac{1}{N_2}  \sum_{i=1}^{N_2}  \tilde C_t^{C^{i}_{\circ}} \sigma'(\tilde Z_t^{C^{i}_{\circ} }(x) )  \tilde W^{2,C^{i}_{\circ},W_{0}, W_0^{2,i} }_t \bigg{|} \notag \\
&+ \bigg{|} \frac{1}{N_2}  \sum_{i=1}^{N_2}  \tilde C_t^{C^{i}_{\circ}} \sigma'(\tilde Z_t^{C^{i}_{\circ} }(x) )  \tilde W^{2,C^{i}_{\circ},W_{0}, W_0^{2,i} }_t - \int_{\mathcal{C}}  \tilde C_t^c  \sigma'(\tilde Z_t^{c}(x) ) \big{(} \int_{\mathcal{W}^2}  \tilde W^{2,c,W_{0},u}_t   \mu_{W^2}(du) \big{)} \mu_c(dc) \bigg{|} \notag \\
&\vcentcolon = | \Gamma_t^{V,1,W_{0}} | +  | \Gamma_t^{V,2,W_{0}} | .
\label{Vbound}
\end{align}

The first term in (\ref{Vbound}) can be studied using a decomposition:
\begin{align}
\bigg{|} \Gamma_t^{V,1,W_{0}} \bigg{|} &=  \bigg{|} \frac{1}{N_2}  \sum_{i=1}^{N_2}  C_t^{C^{i}_{\circ}} \sigma'(Z_t^{C^{i}_{\circ} }(x) )  W^{2,C^{i}_{\circ}, W_{0}, W_0^{2,i}}_t -  \frac{1}{N_2}  \sum_{i=1}^{N_2}  \tilde C_t^{C^{i}_{\circ}} \sigma'(\tilde Z_t^{C^{i}_{\circ} }(x) )  \tilde W^{2,C^{i}_{\circ},W_{0}, W_0^{2,i} }_t \bigg{|} \notag \\
   &=   \bigg{|} \frac{1}{N_2} \sum_{i=1}^{N_2} \bigg{[} \bigg{(} C_t^{C^{i}_{\circ}}  - \tilde C_t^{C^{i}_{\circ}}  \bigg{)}  \sigma'(Z_t^{C^{i}_{\circ} }(x) )  W^{2,C^{i}_{\circ},W_{0}, W_0^{2,i}}_t\nonumber\\
& \qquad
 +   \tilde C_t^{C^{i}_{\circ}}  \bigg{(}  \sigma'(Z_t^{C^{i}_{\circ} }(x) )  W^{2,C^{i}_{\circ},W_{0}, W_0^{2,i}}_t -  \sigma'(\tilde Z_t^{C^{i}_{\circ} }(x) )  \tilde W^{2,C^{i}_{\circ},W_{0}, W_0^{2,i}}_t  \bigg{)}    \bigg{]} \bigg{|} \notag \\
 &= \bigg{|} \frac{1}{N_2} \sum_{i=1}^{N_2} \bigg{[} \bigg{(} C_t^{C^{i}_{\circ}}  - \tilde C_t^{C^{i}_{\circ}}  \bigg{)}  \sigma'(Z_t^{C^{i}_{\circ} }(x) )  W^{2,C^{i}_{\circ},W_{0}, W_0^{2,i}}_t
 +   \tilde C_t^{C^{i}_{\circ}}  \bigg{(}  \sigma'(Z_t^{C^{i}_{\circ} }(x) ) - \sigma'(\tilde Z_t^{C^{i}_{\circ} }(x) ) \bigg{)}  W^{2,C^{i}_{\circ},W_{0}, W_0^{2,i}}_t \notag \\
 & +   \tilde C_t^{C^{i}_{\circ}}   \sigma'(\tilde Z_t^{C^{i}_{\circ} }(x) ) \bigg{(}  W^{2,C^{i}_{\circ},W_{0}, W_0^{2,i}}_t  -  \tilde W^{2,C^{i}_{\circ},W_{0}, W_0^{2,i}}_t   \bigg{)} \bigg{]} \bigg{|} \notag \\
 &\leq \frac{K}{N_2} \sum_{i=1}^{N_2} \bigg{[} \big{|} C_t^{C^{i}_{\circ}}  - \tilde C_t^{C^{i}_{\circ}}  \big{|}   + \big{|} Z_t^{C^{i}_{\circ} }(x) - \tilde Z_t^{C^{i}_{\circ} }(x)  \big{|} +   \big{|}  W^{2,C^{i}_{\circ},W_{0}, W_0^{2,i}}_t  -  \tilde W^{2,C^{i}_{\circ},W_{0}, W_0^{2,i}}_t   \big{|} \bigg{]} \notag \\
 &\leq \frac{K}{N_2} \sum_{i=1}^{N_2}  \sqrt{ E_t^{N_2}(C^{i}_{\circ},W_{0}, W_0^{2,i},x)  }.\notag
\end{align}
where the uniform bounds from (\ref{Eq:UniformBounds}) were used.
 In addition, we also have for some constant $K<\infty$
\begin{align}
\phantom{.} \mathbb{E} \bigg{[}  \bigg{(}  \Gamma_t^{V,2,W_0} \bigg{)}^2 \bigg{]} &= \mathbb{E} \bigg{[}   \bigg{(} \frac{1}{N_2}  \sum_{i=1}^{N_2}  \tilde C_t^{C^{i}_{\circ}} \sigma'(\tilde Z_t^{C^{i}_{\circ} }(x) )  \tilde W^{2,C^{i}_{\circ},W_0, W_0^{2,i} }_t - \int_{\mathcal{C}}  \tilde C_t^c  \sigma'(\tilde Z_t^{c} (x)) \big{(} \int_{\mathcal{W}^2}  \tilde W^{2,c,W_0,u}_t   \mu_{W^2}(du) \big{)} \mu_c(dc) \bigg{)}^2 \bigg{]} \notag \\
&= \frac{1}{N_2^2} \sum_{i=1}^{N_2} \textrm{Var} \big{[}\tilde{C}_t^{C^{i}_{\circ}} \sigma'(\tilde Z_t^{C^{i}_{\circ} }(x) )  \tilde W^{2,C^{i}_{\circ},W_0, W_0^{2,i} }_t  \big{]}\notag \\
& \leq\frac{K }{N_2}.\notag
\end{align}
where we used the assumed independence of $C^{i}_{\circ}$, $W_0$, and $W_0^{2,i}$ as well as the a-priori bound from (\ref{Eq:UniformBounds}).

Hence, we obtain that for some unimportant constant $K<\infty$
\begin{align}
\mathbb{E}| V_t^{N_2, W_{0}}(x) - V_t^{W_{0}}(x) |^{2} &\leq \frac{K}{N_2} \sum_{i=1}^{N_2}  \mathbb{E}\left[ E_t^{N_2}(C^{i}_{\circ},W_{0}, W_0^{2,i},x)\right]  +\frac{K }{N_2}.\label{Eq:BoundV}
\end{align}

\subsection{Bound for $| g_t^{N_2}(x) - g_t(x) |$.}
We can write
\begin{eqnarray}
 | g_t^{N_2}(x) - g_t(x)  | &=& \bigg{|}  \frac{1}{N_2} \sum_{i=1}^{N_2} C_t^{C^{i}_{\circ}} H_t^{2,C^{i}_{\circ}}(x)  - \int_{\mathcal{C}}  \tilde C_t^{c} \tilde H_t^{2,c}(x) \mu_c(dc)  \bigg{|} \notag \\
&\leq&  \bigg{|}  \frac{1}{N_2} \sum_{i=1}^{N_2} C_t^{C^{i}_{\circ}} H_t^{2,C^{i}_{\circ}}(x)  - \frac{1}{N_2} \sum_{i=1}^{N_2} \tilde C_t^{C^{i}_{\circ}} \tilde H_t^{2,C^{i}_{\circ}}(x)   \bigg{|}  \notag \\
& &+ \bigg{|} \frac{1}{N_2} \sum_{i=1}^{N_2} \tilde C_t^{C^{i}_{\circ}} \tilde H_t^{2,C^{i}_{\circ}}(x)    - \int_{\mathcal{C}}  \tilde C_t^{c} \tilde H_t^{2,c}(x) \mu_c(dc)  \bigg{|} \notag \\
&\vcentcolon =& \Gamma_t^{g,1}(x) + \Gamma_t^{g,2}(x).
\label{Gbound}
\end{eqnarray}

Let's analyze the first term in (\ref{Gbound}). Using the uniform bounds from (\ref{Eq:UniformBounds}) we have for some unimportant constant $K<\infty$
\begin{eqnarray}
\bigg{|} \Gamma_t^{g,1}(x) \bigg{|} &=& \bigg{|}  \frac{1}{N_2} \sum_{i=1}^{N_2} C_t^{C^{i}_{\circ}} H_t^{2,C^{i}_{\circ}}(x)  - \frac{1}{N_2} \sum_{i=1}^{N_2} \tilde C_t^{C^{i}_{\circ}} \tilde H_t^{2,C^{i}_{\circ}}(x)   \bigg{|}    \notag \\
&=& \bigg{|} \frac{1}{N_2} \sum_{i=1}^{N_2} (C_t^{C^{i}_{\circ}} - \tilde C_t^{C^{i}_{\circ}} ) H_t^{2,C^{i}_{\circ}}(x)  +  \frac{1}{N_2} \sum_{i=1}^{N_2} \tilde C_t^{C^{i}_{\circ}} (  H_t^{2,C^{i}_{\circ}}(x) - \tilde H_t^{2,C^{i}_{\circ}}(x) ) \bigg{|}  \notag \\
&\leq&  \frac{K}{N_2} \sum_{i=1}^{N_2} \bigg{[} | C_t^{C^{i}_{\circ}} - \tilde C_t^{C^{i}_{\circ}} | + | H_t^{2,C^{i}_{\circ}}(x) - \tilde H_t^{2,C^{i}_{\circ}}(x) | \bigg{]} \notag \\
&\leq&  \frac{K}{N_2} \sum_{i=1}^{N_2}  \sqrt{ E_t^{N_2}(C^{i}_{\circ},W_{0},W_{0}^{2,i},x)  }.\nonumber
\end{eqnarray}

The second term in (\ref{Gbound}) is bounded, as follows,
\begin{align}
\mathbb{E} \bigg{[} \bigg{(}  \Gamma_t^{g,2}(x) \bigg{)}^2 \bigg{]} &= \mathbb{E} \bigg{[} \bigg{(} \frac{1}{N_2} \sum_{i=1}^{N_2} \big{(} \tilde C_t^{C^{i}_{\circ}} \tilde H_t^{2,C^{i}_{\circ}}(x)    - \int_{\mathcal{C}}  \tilde C_t^{c} \tilde H_t^{2,c}(x) \mu_c(dc)  \big{)} \bigg{)}^2 \bigg{]} =  \frac{1}{N_2^2} \sum_{i=1}^{N_2} \textrm{Var} \big{[} \tilde C_t^{C^{i}_{\circ}} \tilde H_t^{2,C^{i}_{\circ}}(x)  \big{]}  \leq \frac{K}{N_2},
\label{Gammag2Bound}
\end{align}
where the independence of $C^{i}_{\circ}$ was used. Hence, putting things together we get
\begin{align}
 \mathbb{E}| g_t^{N_2}(x) - g_t(x)  |&\leq \mathbb{E}\left[\frac{K}{N_2} \sum_{i=1}^{N_2}  \sqrt{ E_t^{N_2}(C^{i}_{\circ},W_{0},W_{0}^{2,i},x)  }\right]+\frac{K}{\sqrt{N_2}}\nonumber.
\end{align}

\subsection{Bound for $\mathbb{E} \bigg{[}( C_t^{C_{\circ}^{i}} - \tilde C_t^{C_{\circ}^{i}} )^2+( W_t^{1,W_{0}} - \tilde W_t^{1,W_{0}} )^2+( W_t^{2,C^{i}_{\circ},W_{0},W_{0}^{2,i}} - \tilde W_t^{2,C^{i}_{\circ},W_{0},W_{0}^{2,i}} )^2\bigg{]} $.}

Let us write for notational convenience $c={C_{\circ}^{i}}$. We have
\begin{align}
d ( C_t^c - \tilde C_t^c )^2 &=  2( C_t^c - \tilde C_t^c ) \int_{\mathcal{X} \times \mathcal{Y} } \bigg{[} \big{(} y -  g_t^{N_2}(x)  \big{)} \big{(} H_t^{2,c}(x) - \tilde H_t^{2,c}(x) \big{)}  + (g_t(x) - g_t^{N_2}(x)  )  \tilde H_t^{2,c}(x)  \bigg{]} \pi(dx,dy)  dt.  \notag
\end{align}

Using Young's inequality and the uniform bounds from (\ref{Eq:UniformBounds}),
\begin{align}
 ( C_t^c - \tilde C_t^c )^2 &\leq  ( C_0^c - \tilde C_0^c )^2 + K  \int_0^t  \int_{\mathcal{X} \times \mathcal{Y} } \pi(dx,dy) \bigg{[} ( C_s^c - \tilde C_s^c )^2 + \big{(} H_s^{2,c}(x) - \tilde H_s^{2,c}(x) \big{)}^2 \notag \\
 & \quad+  | C_s^c - \tilde C_s^c | \cdot | g_s(x) - g_s^{N_2}(x)  | \bigg{]} ds \notag \\
 &= K  \int_0^t  \int_{\mathcal{X} \times \mathcal{Y} } \pi(dx,dy) \bigg{[} ( C_s^c - \tilde C_s^c )^2 + \big{(} H_s^{2,c}(x) - \tilde H_s^{2,c}(x) \big{)}^2 +  | C_s^c - \tilde C_s^c | ( \Gamma_s^{g,1}(x) + \Gamma_s^{g,2}(x) ) \bigg{]} ds \notag \\
  &= K  \int_0^t  \int_{\mathcal{X} \times \mathcal{Y} } \pi(dx,dy) \bigg{[} ( C_s^c - \tilde C_s^c )^2 + \big{(} H_s^{2,c}(x) - \tilde H_s^{2,c}(x) \big{)}^2 +  | C_s^c - \tilde C_s^c |  \Gamma_s^{g,1}(x) +  \big{(} \Gamma_s^{g,2}(x) \big{)}^2 \bigg{]} ds. \notag\\
  \label{Cbound0011A}
\end{align}

For the term $ \displaystyle  | C_s^c - \tilde C_s^c |  \Gamma_s^{g,1}(x)$, where we recall $c=C_{\circ}^{i}$, we have
\begin{eqnarray}
| C_s^{C_{\circ}^{i}} - \tilde C_s^{C_{\circ}^{i}} |  \Gamma_s^{g,1}(x)   &\leq&  | C_s^{C_{\circ}^{i}} - \tilde C_s^{C_{\circ}^{i}} |   \frac{K}{N_2} \sum_{j=1}^{N_2}  \sqrt{ E_s^{N_2}(C_{\circ}^{j},W_{0},W_{0}^{2,j},x)  }  \notag \\
&=& \frac{K}{N_2} \sum_{j=1}^{N_2}    | C_s^{C_{\circ}^{i}} - \tilde C_s^{C_{\circ}^{i}} | \sqrt{ E_s^{N_2}(C_{\circ}^{j},W_{0},W_{0}^{2,j},x)  }   \notag \\
&\leq& K  E_s^{N_2}(C^{i}_{\circ},W_{0},W_{0}^{2,i},x)  +  \frac{K}{N_2} \sum_{j=1}^{N_2}   E_s^{N_2}(C_{\circ}^{j},W_{0},W_{0}^{2,j},x)    \notag
\end{eqnarray}

Therefore, using (\ref{Gammag2Bound}) and (\ref{Cbound0011A}) we obtain
\begin{align}
\mathbb{E} \bigg{[}  ( C_t^{C_{\circ}^{i}} - \tilde C_t^{C_{\circ}^{i}} )^2 \bigg{]} &\leq  \frac{K_1}{N_2} +  K_2 \int_0^t  \sup_{x \in \mathcal{X}} \mathbb{E} \bigg{[}   E_s^{N_2}(C^{i}_{\circ},W_{0},W_{0}^{2,i},x) \bigg{]} ds\notag\\
& +\frac{K_{3}}{N_{2}}\sum_{j=1}^{N_{2}} \int_0^t  \sup_{x \in \mathcal{X}} \mathbb{E} \bigg{[}   E_s^{N_2}(C_{\circ}^{j},W_{0},W_{0}^{2,j},x) \bigg{]} ds\nonumber
\end{align}

Using similar arguments, we can also show, using (\ref{Eq:BoundV}), that
\begin{align}
\mathbb{E} \bigg{[}  ( W_t^{1,W_{0}} - \tilde W_t^{1,W_{0}} )^2 \bigg{]} &\leq \frac{K_1}{N_2} + K_2 \int_0^t  \sup_{x \in \mathcal{X}} \mathbb{E} \left[   E_s^{N_2}(C^{i}_{\circ},W_{0},W_{0}^{2,i},x) \right]ds\nonumber\\
&\qquad\qquad+ \frac{K_{3}}{N_{2}}\sum_{j=1}^{N_{2}} \int_0^t  \sup_{x \in \mathcal{X}} \mathbb{E} \left[  E_s^{N_2}(C_{\circ}^{j},W_{0},W_{0}^{2,j},x)  \right] ds\nonumber\\
 \mathbb{E}  \bigg{[} ( W_t^{2,C^{i}_{\circ},W_{0},W_{0}^{2,i}} - \tilde W_t^{2,C^{i}_{\circ},W_{0},W_{0}^{2,i}} )^2 \bigg{]} &\leq \frac{K_1}{N_2}+ K_{2} \int_0^t  \sup_{x \in \mathcal{X}} \mathbb{E}\left[  E_s^{N_2}(C^{i}_{\circ},W_{0},W_{0}^{2,i},x)\right] ds\nonumber\\
 &\qquad\qquad+\frac{K_{3}}{N_{2}}\sum_{j=1}^{N_{2}} \int_0^t  \sup_{x \in \mathcal{X}} \mathbb{E}\left[   E_s^{N_2}(C_{\circ}^{j},W_{0},W_{0}^{2,j},x)  \right] ds
\label{WTildeBounds0}
\end{align}

Therefore, we overall get that
\begin{align}
&\mathbb{E}\left[( C_t^{C_{\circ}^{i}} - \tilde C_t^{C_{\circ}^{i}} )^2+( W_t^{1,W_{0}} - \tilde W_t^{1,W_{0}} )^2+( W_t^{2,C^{i}_{\circ},W_{0},W_{0}^{2,i}} - \tilde W_t^{2,C^{i}_{\circ},W_{0},W_{0}^{2,i}} )^2\right]\leq \nonumber\\
&\quad\leq\frac{K_1}{N_2} + K_2 \int_0^t  \sup_{x \in \mathcal{X}} \mathbb{E} \left[   E_s^{N_2}(C^{i}_{\circ},W_{0},W_{0}^{2,i},x) \right]ds+ \frac{K_{3}}{N_{2}}\sum_{j=1}^{N_{2}} \int_0^t  \sup_{x \in \mathcal{X}} \mathbb{E} \left[  E_s^{N_2}(C_{\circ}^{j},W_{0},W_{0}^{2,j},x)  \right] ds\label{Eq:Bound1}
\end{align}

\subsection{Bound for $\mathbb{E} \bigg{[}(Z_t^{C_{\circ}^{i}}(x) - \tilde Z_t^{C_{\circ}^{i}}(x) )^2 \bigg{]}$.}

We next consider $(Z_t^{C_{\circ}^{i}}(x) - \tilde Z_t^{C_{\circ}^{i}}(x))^2$. For the following calculations, we define $\mathcal{F}_C^{N_2} = ( C^1_{\circ}, C_{\circ}^2, \ldots, C^{N_2}_{\circ})$.

For $i =1, 2, \ldots, N_2$, we have
\begin{align}
( Z_t^{C^{i}_{\circ}}(x)- \tilde Z_t^{C^{i}_{\circ}}(x) )^2 &= \bigg{(}    \mathbb{E} \bigg{[}   W^{2,C^{i}_{\circ},W_{0}, W_0^{2,i} }_t  H^{1,W_{0}}_t(x) \bigg{|} \mathcal{F}_C^{N_2} \bigg{]}  -    \int_{\mathcal{W}^{1}} \int_{\mathcal{W}^2}   \tilde  W^{2,C^{i}_{\circ},w,u }_t  \tilde H^{1,w}_t(x)  \mu_{W^2}(du) \mu_{W^{1}}(dw) \bigg{)}^2  \notag \\
&= \bigg{(}    \mathbb{E} \bigg{[}  ( W^{2,C^{i}_{\circ},W_{0}, W_0^{2,i} }_t -   \tilde  W^{2,C^{i}_{\circ},W_{0}, W_0^{2,i} }_t  )  H^{1,W_{0}}_t(x) \bigg{|} \mathcal{F}_C^{N_2} \bigg{]}  \notag \\
&+  \mathbb{E} \bigg{[}   \tilde W^{2,C^{i}_{\circ},W_{0}, W_0^{2,i} }_t  H^{1,W_{0}}_t(x) \bigg{|} \mathcal{F}_C^{N_2} \bigg{]} -    \int_{\mathcal{W}^{1}} \int_{\mathcal{W}^2}   \tilde  W^{2,C^{i}_{\circ},w,u }_t  \tilde H^{1,w}_t(x)  \mu_{W^2}(du) \mu_{W^{1}}(dw) \bigg{)}^2 \notag \\
&\leq    \mathbb{E} \bigg{[}  ( W^{2,C^{i}_{\circ},W_{0}, W_0^{2,i} }_t -   \tilde  W^{2,C^{i}_{\circ},W_{0}, W_0^{2,i} }_t  )^2 \bigg{|} \mathcal{F}_C^{N_2} \bigg{]} \notag \\
&+  \left(\mathbb{E} \bigg{[}   \big{[}  \tilde W^{2,C^{i}_{\circ},W_{0}, W_0^{2,i} }_t  H^{1,W_0}_t(x) -   \int_{\mathcal{W}^{1}} \int_{\mathcal{W}^2} \tilde  W^{2,C^{i}_{\circ},w,u }_t  \tilde H^{1,w}_t(x)  \mu_{W^2}(du) \mu_{W^{1}}(dw)  \big{]} \bigg{|} \mathcal{F}_C^{N_2} \bigg{]}\right)^2 \notag \\
&\leq \mathbb{E} \bigg{[}  ( W^{2,C^{i}_{\circ},W_{0}, W_0^{2,i} }_t -   \tilde  W^{2,C^{i}_{\circ},W_{0}, W_0^{2,i} }_t  )^2 \bigg{|} \mathcal{F}_C^{N_2} \bigg{]} \notag \\
&+  \left(\mathbb{E} \bigg{[}   \big{[}  \tilde W^{2,C^{i}_{\circ},W_{0}, W_0^{2,i} }_t  H^{1,W_{0}}_t(x) -   \tilde  W^{2,C^{i}_{\circ},W_{0},W_0^{2,i} }_t  \tilde H^{1,W_{0}}_t(x)  \big{]} \bigg{|} \mathcal{F}_C^{N_2} \bigg{]}\right)^2 \notag \\
&\leq \mathbb{E} \bigg{[}  ( W^{2,C^{i}_{\circ},W_{0}, W_0^{2,i} }_t -   \tilde  W^{2,C^{i}_{\circ},W_{0}, W_0^{2,i} }_t  )^2 \bigg{|} \mathcal{F}_C^{N_2} \bigg{]} + K_0  \mathbb{E} \bigg{[} ( W_t^{1,W_{0}} - \tilde W_t^{1,W_{0}} )^2  \bigg{|} \mathcal{F}_C^{N_2} \bigg{]},  \notag
\end{align}
where the uniform bounds from (\ref{Eq:UniformBounds}) were used together with the compactness of the state space assumption from Assumption \ref{A:Assumption1}.

Therefore, using iterated expectations, we have that
\begin{align}
\mathbb{E} \bigg{[}  (  Z_t^{C^{i}_{\circ}}(x)- \tilde Z_t^{C^{i}_{\circ}}(x)) ^2  \bigg{]} &\leq \mathbb{E} \bigg{[}  ( W^{2,C^{i}_{\circ},W_{0}, W_0^{2,i} }_t -   \tilde  W^{2,C^{i}_{\circ},W_{0}, W_0^{2,i} }_t  )^2  \bigg{]} + K_0  \mathbb{E} \bigg{[} ( W_t^{1,W_{0}} - \tilde W_t^{1,W_{0}} )^2   \bigg{]}\nonumber\\
&\leq \frac{K_1}{ N_2}+ K_{2} \int_0^t  \sup_{x \in \mathcal{X}} \mathbb{E}\left[  E_s^{N_2}(C^{i}_{\circ},W_{0},W_{0}^{2,i},x)\right] ds\nonumber\\
 &\qquad\qquad+\frac{K_{3}}{N_{2}}\sum_{j=1}^{N_{2}} \int_0^t  \sup_{x \in \mathcal{X}} \mathbb{E}\left[   E_s^{N_2}(C_{\circ}^j,W_{0},W_{0}^{2,j},x)  \right] ds\label{Eq:Bound2}
\end{align}
where (\ref{WTildeBounds0})  was used.

Finally, using the assumption that $\sigma(\cdot)$ is globally Lipschitz,
\begin{align}
\mathbb{E} \bigg{[}  \big{(} H_t^{2,C^{i}_{\circ}}(x) - \tilde H_t^{2,C^{i}_{\circ}}(x) \big{)} ^2 \bigg{]} &\leq K_0 \mathbb{E} \bigg{[}  (  Z_t^{C^{i}_{\circ}}- \tilde Z_t^{C^{i}_{\circ}}) ^2  \bigg{]}\nonumber\\
&\leq\frac{K_1}{ N_2}+ K_{2} \int_0^t  \sup_{x \in \mathcal{X}} \mathbb{E}\left[  E_s^{N_2}(C^{i}_{\circ},W_{0},W_{0}^{2,i},x)\right] ds\nonumber\\
 &\qquad\qquad+\frac{K_{3}}{N_{2}}\sum_{j=1}^{N_{2}} \int_0^t  \sup_{x \in \mathcal{X}} \mathbb{E}\left[   E_s^{N_2}(C_{\circ}^j,W_{0},W_{0}^{2,j},x)  \right] ds,\label{Eq:Bound3}
\end{align}
for possibly different, but finite constants $K_1,K_2,K_3<\infty$.

\subsection{Bound for $\mathbb{E}\left[E_t^{N_2}(C^{i}_{\circ},W_{0},W_{0}^{2,i},x)\right]$}

Collecting our results from (\ref{Eq:Bound1}), (\ref{Eq:Bound2}) and (\ref{Eq:Bound3}), together with the definition of the error function $E_t^{N_2}$, we have, for $i=1,\cdots, N_{2}$, the bound
\begin{align}
\sup_{x \in \mathcal{X}} \mathbb{E}\left[ E_t^{N_2}(C^{i}_{\circ},W_{0},W_{0}^{2,i},x) \right]  &\leq \frac{K_1}{ N_2} + K_2 \int_0^t  \sup_{x \in \mathcal{X}} \mathbb{E} \left[   E_s^{N_2}(C^{i}_{\circ},W_{0},W_{0}^{2,i}, x) \right]ds,\nonumber\\
&\quad+\frac{K_{3}}{N_{2}}\sum_{j=1}^{N_{2}} \int_0^t  \sup_{x \in \mathcal{X}} \mathbb{E}\left[   E_s^{N_2}(C_{\circ}^j,W_{0},W_{0}^{2,j},x)  \right] ds\nonumber
\end{align}

Therefore, by averaging over all $i=1,\cdots, N_{2}$ and then using Gronwall's inequality, we get for any $0 \leq t \leq T$,
\begin{eqnarray}
\frac{1}{N_{2}}\sum_{i=1}^{N_{2}} \sup_{x \in \mathcal{X}} \mathbb{E} \left[  E_t^{N_2}(C^{i}_{\circ},W_{0},W_{0}^{2,i},x) \right] &\leq&  \frac{K}{ N_2}.\nonumber
\end{eqnarray}
for an appropriate constant $K<\infty$. Combining the last two displays we naturally get, again using Gronwall's inequality, that for $i=1,\cdots, N_{2}$
\begin{eqnarray}
\sup_{x \in \mathcal{X}}  \mathbb{E} \left[  E_t^{N_2}(C^{i}_{\circ},W_{0},W_{0}^{2,i},x) \right] &\leq&  \frac{K}{ N_2}.
\label{Einequality}
\end{eqnarray}
where the constant $K<\infty$ does not depend on $i$ or $N_2$.
\subsection{Convergence of neural network prediction}

The bound (\ref{Einequality}) of course proves the (uniform) convergence in probability of the neural network output $g_t^{N_2}(x) \rightarrow g_t(x)$. Recall, by (\ref{Gbound}), that
\begin{align}
 \mathbb{E}| g_t^{N_2}(x) - g_t(x)  |&\leq \mathbb{E} \bigg{|} \Gamma_t^{g,1}(x) \bigg{|} + \frac{K}{\sqrt{N_2}}\nonumber.
\end{align}
where
\begin{eqnarray}
\bigg{|} \Gamma_t^{g,1}(x) \bigg{|} \leq  \frac{K}{N_2} \sum_{i=1}^{N_2}  \sqrt{ E_t^{N_2} (C^{i}_{\circ},W_{0},W_{0}^{2,i},x)  }.\nonumber
\end{eqnarray}

Then, using the Cauchy-Schwartz inequality and (\ref{Einequality}),
\begin{eqnarray}
\mathbb{E} \bigg{|} \Gamma_t^{g,1}(x) \bigg{|} &\leq&  \frac{K}{N_2} \sum_{i=1}^{N_2}  \mathbb{E} \bigg{[} \sqrt{ E_t^{N_2}(C^{i}_{\circ},W_{0},W_{0}^{2,i},x)  } \bigg{]} \notag \\
&\leq& \frac{K}{N_2} \sum_{i=1}^{N_2} \sqrt{ \mathbb{E} \bigg{[}  E_t^{N_2}(C^{i}_{\circ},W_{0},W_{0}^{2,i},x)   \bigg{]} }  \notag \\
&\leq& \frac{K}{N_2} \sum_{i=1}^{N_2} \sqrt{ \sup_{x \in \mathcal{X}} \mathbb{E} \bigg{[}  E_t^{N_2}(C^{i}_{\circ},W_{0},W_{0}^{2,i},x)   \bigg{]} }  \notag \\
&\leq& \frac{K}{ \sqrt{N_2}}.\nonumber
\end{eqnarray}

Therefore, for $0 \leq t \leq T$, and for some unimportant constant $K<\infty$
\begin{eqnarray}
\sup_{x \in \mathcal{X}} \mathbb{E} \bigg{[} | g_t^{N_2}(x) - g_t(x) | \bigg{]} \leq \frac{K}{ \sqrt{N_2}}, \nonumber
\end{eqnarray}
concluding the identification of the limit in Theorem \ref{Theorem1}.

\section{Proof of Theorem \ref{Theorem1}-Uniqueness of the limit}\label{S:PropertiesLimitingSystem}

\begin{lemma}
The solution to the limiting system (\ref{Eq:LLN_DNNTheoremIntro}) is unique in $C([0,1], \mathcal{C} \times \mathcal{W}^{1} \times \mathcal{W}^2 \times \mathcal{X})$.
\end{lemma}
\begin{proof}

Suppose there are two solutions to (\ref{Eq:LLN_DNNTheoremIntro}). Let's denote the first solution as $(\hat W_t^{1,w}, \hat W^{2,c,w,u}_{t}, \hat C_t^c, \hat Z_t^c(x), \hat V_t^w(x), \hat g_t(x))$ and the second solution as $(\bar W_t^{1,w}, \bar W^{2,c,w,u}_{t}, \bar C_t^c, \bar Z_t^c(x), \bar V_t^w(x), \bar g_t(x))$. Define the function

\begin{align}
Q_t &=  \sup_{(c,w,u, x) \in \mathcal{C} \times \mathcal{W}^1 \times \mathcal{W}^2  \times \mathcal{X} }\left\{ ( \hat W^{2,c,w,u}_t - \bar W^{2,c,w,u}_t )^2+  \norm{ \hat W_t^{1,w} - \bar W_t^{1,w} }^2 + ( \hat V_t^w(x) - \bar V_t^{w}(x))^2 + ( \hat Z_t^c(x) - \bar Z_t^c(x) )^2 \right.\notag \\
& \left.\hspace{4cm} + ( \hat C_t^c - \bar C_t^c )^2 + ( \hat g_t(x) - \bar g_t(x) )^2 \right\}\notag
\end{align}
Note that
\begin{eqnarray}
Q_0 = 0.\notag
\end{eqnarray}
We next study the evolution of $Q_t$ for $t > 0$.

Using the same approach as in Section \ref{AprioriBoundsSecondLayer}, we can show that any $\tilde W^{2,c,w,u}_{t}, \tilde C_t^c, \tilde W_t^{1,w}, \tilde H_t^{1,w}(x), \tilde Z_t^c, V_t^w,$ and $g_t(x)$ which solve (\ref{Eq:LLN_DNNTheoremIntro}) are uniformly bounded on $\mathcal{C} \times \mathcal{W}^{1} \times \mathcal{W}^2 \times \mathcal{X} \times [0,1]$.

We can then prove the inequality
\begin{eqnarray}
( \hat W^{2,c,w,u}_{t}  - \bar W^{2,c,w,u}_{t} )^2&=& 2\int_0^t  \int_{\mathcal{X} \times \mathcal{Y} }( \hat W^{2,c,w,u}_{s}  - \bar W^{2,c,w',u}_{s} )  \bigg{(} \big{(} y -  \hat g_s(x) \big{)} \hat C_s^{c} \sigma'(\hat Z_s^{c}(x)) )  \hat H_s^{1,w}(x)  \notag \\
& &\quad- \big{(} y - \bar g_s(x) \big{)} \bar C_s^{c} \sigma'(\bar Z_s^{c}(x) )  \bar H_s^{1,w}(x)   \bigg{)}  \pi(dx,dy) ds \notag \\
&\leq& K \int_0^t Q_s ds.\notag
\end{eqnarray}
where we have also used Young's inequality and the fact that $\mathcal{X} \times \mathcal{Y}$ is compact. Therefore,
\begin{eqnarray}
 \sup_{(c,w,u, x) \in \mathcal{C} \times \mathcal{W}^1 \times \mathcal{W}^2  \times \mathcal{X} } ( \hat W^{2,c,w,u}_{t}  - \bar W^{2,c,w,u}_{t} )^2 \leq K \int_0^t Q_s ds.\notag
\end{eqnarray}
Similarly,
\begin{eqnarray}
 \sup_{(c,w,u, x) \in \mathcal{C} \times \mathcal{W}^1 \times \mathcal{W}^2  \times \mathcal{X} } \norm{ \hat W^{1,w}_{t}  - \bar W^{1,w}_{t} }^2 &\leq& K \int_0^t Q_s ds.  \notag \\
  \sup_{(c,w,u, x) \in \mathcal{C} \times \mathcal{W}^1 \times \mathcal{W}^2  \times \mathcal{X} } ( \hat C_t^c  - \bar C_t^c )^2 &\leq& K \int_0^t Q_s ds.\notag
\end{eqnarray}

Next, using the Cauchy-Schwarz inequality and Young's inequality,
\begin{align}
( \hat Z_t^c(x) -  \bar Z_t^c (x))^2 &= \bigg{(} \int_{\mathcal{W}^{1}} \int_{\mathcal{W}^2}   \hat  W^{2,c,w,u }_t  \hat H^{1,w}_t(x)  \mu_{W^2}(du) \mu_{W^{1}}(dw)  -\nonumber\\
&\hspace{4cm}-\int_{\mathcal{W}^{1}} \int_{\mathcal{W}^2}   \bar W^{2,c,w,u }_t  \bar H^{1,w}_t(x)  \mu_{W^2}(du) \mu_{W^{1}}(dw)  \bigg{)}^2 \notag \\
&= \bigg{(} \int_{\mathcal{W}^{1}} \int_{\mathcal{W}^2} (   \hat  W^{2,c,w,u }_t  \hat H^{1,w}_t(x)  -   \bar W^{2,c,w,u }_t  \bar H^{1,w}_t(x) )  \mu_{W^2}(du)    \mu_{W^{1}}(dw) \bigg{)}^2 \notag \\
&\leq K \int_{\mathcal{W}^{1}} \int_{\mathcal{W}^2} \bigg{[} (   \hat  W^{2,c,w,u }_t  -    \bar  W^{2,c,w,u }_t   ) + ( \hat W_t^{1,w} - \bar W_t^{1,w} )^2  \bigg{]} \mu_{W^2}(du)    \mu_{W^{1}}\hat{}(dw)  \notag \\
&\leq K  \sup_{(c,w,u, x) \in \mathcal{C} \times \mathcal{W}^1 \times \mathcal{W}^2  \times \mathcal{X} } \bigg{[}  (   \hat  W^{2,c,w,u }_t  -    \bar  W^{2,c,w,u }_t   ) + ( \hat W_t^{1,w} - \bar W_t^{1,w} )^2  \bigg{]}  \notag \\
&\leq K \int_0^t Q_s ds.\notag
\end{align}

Using a similar approach,
\begin{align}
( \hat V_t^w (x)-  \bar V_t^w (x))^2 &= \bigg{(}  \int_{\mathcal{C}}   \int_{\mathcal{W}^2} \hat C_t^c  \sigma'(\hat Z_t^{c}(x) )   \hat W^{2,c,w,u}_t   \mu_{W^2}(du)  \mu_c(dc) \notag \\
& \quad-  \int_{\mathcal{C}} \int_{\mathcal{W}^2}  \bar C_t^c  \sigma'(\bar Z_t^{c}(x) )   \bar W^{2,c,w,u}_t   \mu_{W^2}(du) \mu_c(dc) \bigg{)}^2 \notag \\
&\leq   \int_{\mathcal{C}}   \int_{\mathcal{W}^2} \bigg{(} \hat C_t^c  \sigma'(\hat Z_t^{c}(x) )   \hat W^{2,c,w,u}_t  - \bar C_t^c  \sigma'(\bar Z_t^{c}(x) )   \bar W^{2,c,w,u}_t   \bigg{)}^2  \mu_{W^2}(du) \mu_c(dc) \notag \\
&\leq  K \int_{\mathcal{C}}   \int_{\mathcal{W}^2} \bigg{[} ( \hat C_t^c  - \bar C_t^c )^2 + ( \hat Z_t^c(x) - \bar Z_t^{c}(x) )^2 + ( \hat  W^{2,c,w,u}_t -   \bar W^{2,c,w,u}_t )^2  \bigg{]} \mu_{W^2}(du) \mu_c(dc) \notag \\
&\leq  K \sup_{(c,w,u, x) \in \mathcal{C} \times \mathcal{W}^1 \times \mathcal{W}^2  \times \mathcal{X} } \bigg{[}   ( \hat C_t^c  - \bar C_t^c )^2 + ( \hat Z_t^c(x) - \bar Z_t^{c}(x) )^2 + ( \hat  W^{2,c,w,u}_t -   \bar W^{2,c,w,u}_t )^2  \bigg{]}  \notag \\
&\leq K \int_0^t Q_s ds.\notag
\end{align}

Therefore,
\begin{eqnarray}
\sup_{(c,w,u, x) \in \mathcal{C} \times \mathcal{W}^1 \times \mathcal{W}^2  \times \mathcal{X} }  ( \hat Z_t^c (x)-  \bar Z_t^c (x))^2 &\leq& K \int_0^t Q_s ds, \notag \\
\sup_{(c,w,u, x) \in \mathcal{C} \times \mathcal{W}^1 \times \mathcal{W}^2  \times \mathcal{X} }  ( \hat V_t^w (x)-  \bar V_t^w (x))^2 &\leq& K \int_0^t Q_s ds.\notag
\end{eqnarray}

Finally,
\begin{eqnarray}
( \hat g_t(x) -  \bar g_t(x) )^2 &=& \bigg{(} \int_{\mathcal{C}}  \hat C_t^{c} \hat H_t^{2,c}(x) \mu_c(dc) - \int_{\mathcal{C}}  \bar C_t^{c} \bar H_t^{2,c}(x) \mu_c(dc) \bigg{)}^2 \notag \\
&\leq& \int_{\mathcal{C}}   \bigg{[} \hat C_t^{c} \hat H_t^{2,c}(x) \mu_c(dc) - \bar C_t^{c} \bar H_t^{2,c}(x) \bigg{]}^2 \mu_c(dc) \notag \\
&\leq& K \int_{\mathcal{C}}   \bigg{[} ( \hat C_t^{c} - \bar C_t^{c} )^2 +  ( \hat Z_t^c(x) - \bar Z_t^c(x) )^2 \bigg{]} \mu_c(dc) \notag \\
&\leq& K \int_0^t Q_s ds.\notag
\end{eqnarray}
Consequently,
\begin{eqnarray}
\sup_{(c,w,u, x) \in \mathcal{C} \times \mathcal{W}^1 \times \mathcal{W}^2  \times \mathcal{X} }  ( \hat g_t (x)-  \bar g_t (x))^2 \leq K \int_0^t Q_s ds.\notag
\end{eqnarray}

Collecting our results,
\begin{eqnarray}
Q_t &\leq& K \int_0^t Q_s ds, \notag \\
Q_0 &=& 0.\notag
\end{eqnarray}
Therefore, by Gronwall's inequality,
\begin{eqnarray}
Q_t = 0,\notag
\end{eqnarray}
for all $t \in [0,1]$, completing the proof of the lemma.

\end{proof}

\section{Conclusions and future work}\label{S:Conclusions}

In this paper, we have developed an approach that allows us to obtain the limiting behavior of multi-layer neural networks in the mean-field scaling as the number of units per layer and stochastic gradient steps grow to infinity. We have demonstrated that the limit is characterized by paths of weights that connecting the input layer to the output layer. The limit procedure shows that the main hyperparameters (the learning rates) must be chosen using a specific scaling (as a function of the number of hidden units per layer) in order to get a well-defined limit. A numerical study demonstrates that the mean-field scaling of the learning rates has a significant effect on how well the network can be trained. In addition, we have shown that the limit neural network seeks to minimize the limit objective function.

There  are of course a number of open questions. The discussion in Section \ref{ChallengesIntro} shows that a genuine mean field formulation in terms of convergence of empirical distributions of trained parameters has yet to be understood. Secondly, our methodology works for convex error functions like the regression task. It is of interest to characterize the limit in other cases such as Wassenstein distances or regularized entropies. Thirdly, the universal approximation theorems say that a neural network approximator in the infinite unit per layer limit can approximate any continuous function \cite{Barron,Funakashi1989,Hornik1,Hornik2}. It would be of interest to investigate whether the limit procedure developed in this paper could lead to more specific information on convergence rates. Fourthly, the, commonly used in practice, unregularized algorithm and its finer properties (such as more detailed behavior  as time grows to infinity) need to be better understood and we hope that the results in this paper build towards this direction.

\appendix

\section{Limit of First Layer, Lemma \ref{LemmaFirstLayermaintext}}  \label{ProofFirstLayer}

In this appendix we prove Lemma \ref{LemmaFirstLayermaintext}, which is about  the limit of the first layer. The proof is analogous to \cite{NeuralNetworkLLN}. Hence, instead of repeating all the arguments we will only present the general outline emphasizing the differences and refer the interested reader to \cite{NeuralNetworkLLN} when the arguments are the same.

We let $N_1 \rightarrow \infty$ (with $N_2$ fixed). We want to prove that for each $N_2 \in \mathbb{N}$, $\tilde \gamma_{\cdot}^{N_1, N_2}  \overset{d} \rightarrow \gamma_{\cdot}^{N_2}$ in $D_E([0,1])$. $\gamma_{t}^{N_2}$ is a random measure-valued process which, for every $f\in C^{2}_{b}(\mathbb{R}^{d + 2 N_2})$, $\gamma^{N_2}$, satisfies the evolution equation (\ref{LimitSPDELLNmaintext})-(\ref{LimitSPDELLNmaintextb}).

\subsection{Relative Compactness} \label{RelativeCompactness}

Let's first establish a bound on $C_k^i$. Recall that $\sigma(\cdot)$ is bounded and $\alpha_C^{N_1,N_2} =  \frac{N_2}{N_1}$. In the following calculations, the unimportant constants, $K, K_0, K_1,$ and $K_2$ may change from line to line.

For $k = 0, 1, \ldots, \floor*{ N_1}$,
\begin{eqnarray}
C^i_{k+1} &=& C^i_{k} + \frac{\alpha_C^N }{N_2} \bigg{(} y_k - \frac{1}{N_2} \sum_{m=1}^{N_2} C^m_k H^{2,m}_k(x_k) \bigg{)} H_k^{2,i}(x_k), \notag
\end{eqnarray}
Since $\sigma(\cdot)$ is bounded, $| H^{2,i}_k | < K$. Therefore,

\begin{eqnarray}
\big{|} C^i_{k+1} \big{|}  \leq  \big{|} C^i_{k}  \big{|} + \frac{1}{N_1} \bigg{(} K_1 + \frac{1}{N_2} \sum_{m=1}^{N_2}  \big{|} C^m_k \big{|}  \bigg{)}. \notag
\end{eqnarray}
This yields
\begin{eqnarray}
\frac{1}{N_2} \sum_{i=1}^{N_2} \big{|} C^i_{k+1} \big{|}  \leq \frac{1}{N_2} \sum_{i=1}^{N_2}   \big{|} C^i_{k}  \big{|} + \frac{1}{N_1} \bigg{(} K_1 + \frac{1}{N_2} \sum_{i=1}^{N_2}  \big{|} C^i_k \big{|}  \bigg{)}. \notag
\end{eqnarray}

This implies that
\begin{eqnarray}
\frac{1}{N_2} \sum_{i=1}^{N_2} \big{|} C^i_{k+1} \big{|} &\leq& \frac{1}{N_2} \sum_{i=1}^{N_2}   \big{|} C^i_{0}  \big{|} + K_1 \frac{k}{N_1} + \frac{K_2}{N_1} \sum_{j=1}^k \frac{1}{N_2} \sum_{i=1}^{N_2}  \big{|} C^i_k \big{|} \notag \\
&\leq& K_0 + K_1 \frac{k}{N_1} + \frac{K_2}{N_1} \sum_{j=1}^k \frac{1}{N_2} \sum_{i=1}^{N_2}  \big{|} C^i_k \big{|} \notag
\end{eqnarray}

By Gronwall's inequality, for $k \leq \floor*{N_1}$,
\begin{eqnarray}
\frac{1}{N_2} \sum_{i=1}^{N_2} \big{|} C^i_{k} \big{|} \leq K \exp( K).\nonumber
\end{eqnarray}

Then, we also have that
\begin{eqnarray}
\big{|} C^i_{k+1} \big{|} & \leq &  \big{|} C^i_{k}  \big{|} + \frac{K}{N_1}. \notag \\
&\leq & K \frac{k}{N_1}.\nonumber
\end{eqnarray}

This immediately yields that for $k \leq \floor*{N_1}$
\begin{eqnarray}
\big{|} C^i_{k} \big{|} < K.\nonumber
\end{eqnarray}

Let's now address the parameters $W_k^{2,i,j}$. Recall that $\alpha_{W,2}^{N_1,N_2} = N_2$. Then,
\begin{eqnarray}
\big{|} W^{2,i,j}_{k+1} \big{|} &\leq& \big{|} W^{2,i,j}_{k} \big{|}  + \frac{ \alpha_{W,2}^{N_1,N_2}}{ N_1 N_2 } \bigg{|} \big{(} y_k - g^N_{\theta_k}(x_k) \big{)} C_k^i \sigma'(Z_k^{i} )  H^{1,i}_k \bigg{|} \notag \\
&\leq& \big{|} W^{2,i,j}_{k} \big{|} + \frac{K}{N_1}.\nonumber
\end{eqnarray}

Therefore, for $k \leq \floor*{N_1}$,
\begin{eqnarray}
\big{|} W^{2,i,j}_{k+1} \big{|} \leq K.\nonumber
\end{eqnarray}

Similarly, we have
\begin{eqnarray}
\big{|} W^{1,j}_{k+1} \big{|} &\leq& \big{|} W^{1,j}_{k}  \big{|} +  \frac{1}{N_1} \big{|} \big{(} y_k - g^N_{\theta_k}(x_k)  \big{)} \big{(} \frac{1}{N_2} \sum_{i=1}^{N_2} C_k^i \sigma'(Z_k^{i} )  W^{2,i,j}_k \big{)}  \sigma'(W_k \cdot x ) x_j \big{|} \notag \\
&\leq& \big{|} W^{1,j}_{k}  \big{|} +  \frac{K}{N_1}.\nonumber
\end{eqnarray}

Therefore, we obtain
\begin{eqnarray}
\big{|} W^{1,j}_{k} \big{|} \leq K. \nonumber
\end{eqnarray}

Collecting our results, for all $k/N_1\leq 1$ and for all $j=1,\cdots, N_2$, we have the uniform bound
\begin{eqnarray}
| C^i_k | + \norm{ W^{1,j}_k } + | W_k^{2,i,j} |  < K.
\label{PreLimitUniformBound}
\end{eqnarray}

We now prove relative compactness of the family $\{\gamma^{N_1, N_2}\}_{N_1 \in \mathbb{N}}$ in $D_E([0,1])$ where $E = \mathcal{M} ( \mathbb{R}^{d+2N_2})$.  It is sufficient to show compact containment and regularity of the $\gamma^{N_1, N_2}$'s (see for example Chapter 3 of \cite{EthierAndKurtz}).

\begin{lemma}\label{L:CompactContainment}
For each $\eta > 0$ and $t \geq 0$, there is a compact subset $\mathcal{K}$ of E such that
\begin{eqnarray*}
\sup_{N_1 \in \mathbb{N}, 0 \leq t \leq T} \mathbb{P}[ \gamma_t^{N_1, N_2} \notin \mathcal{K} ] < \eta.\nonumber
\end{eqnarray*}
\end{lemma}
\begin{proof}
This uniform bound (\ref{PreLimitUniformBound}) actually implies the stronger statement of compact support. In particular, notice that the set $[-K,K]^{d+2N_2}$ is compact, and define
\begin{equation*}\label{E:KDef} \mathcal{K} = \left\{ \omega\in M(\mathbb{R}^{d+2N_2}):\, \omega\left([-K,K]^{d+2N_2}\right) = 1\right\}.
 \end{equation*}

Then $\mathcal{K}\subset M(\mathbb{R}^{d+2N_2})$, and $\mathbb{P}$-a.s. $\gamma_{t}^{N_1, N_2} \in \mathcal{K}$
for all $N_1 \in \mathbb{N}$ and $t\in[0,1]$. This concludes the proof.
\end{proof}

We now establish regularity of the $\gamma^{N_1, N_2}$'s. Define the function $q(z_{1},z_{2})=\min\{|z_{1}-z_{2}|,1\}$ where $z_{1},z_{2} \in \mathbb{R}$.
\begin{lemma}\label{L:regularity}
There is a constant $K<\infty$ such that for $0\leq u\leq \delta$,  $0\leq v\leq \delta\wedge t$, $t\in[0,1]$,
\begin{equation*}
 \mathbb{E}\left[q(\left< f,\gamma^{N_1,N2}_{t+u}\right>,\left< f, \gamma^{N_1, N_2}_t\right>)q(\left< f, \gamma^{N_1, N_2}_t\right>,\left< f,\gamma^{N_1, N_2}_{t-v}\right>)\big| \mathcal{F}^N_t\right]  \le  C \delta + \frac{K}{N_1}.
\end{equation*}

\end{lemma}

\begin{proof}

Let $0\leq s<t<\leq T$ and $\delta<1$, such that $t-s<\delta<1$. We then have
\begin{eqnarray}
|  C^i_{\floor*{N_1 t}} -  C^i_{\floor*{N_1 s}} | &=& | \sum_{k = \floor*{N_1 s } }^{\floor*{N_1 t}-1} ( C_{k+1}^i - C_k^i  ) | \notag \\
&\leq&  \sum_{k = \floor*{N s } }^{\floor*{N_1 t}-1} \frac{1}{N_1}  \big{|} \big{(} y_k - g^{N_1, N_2}_{\theta_k}(x_k)  \big{)} H_k^{2,i}(x_k) \big{|} \notag \\
&\leq& \frac{1}{N_1} \sum_{k = \floor*{N_1 s } }^{\floor*{N_1 t}-1} K  \notag\\
&\leq& K_1 \delta + \frac{K_2}{N_1}. \notag
\end{eqnarray}

Using the same approach, we can establish similar bounds for the other parameters:
\begin{eqnarray}
|  W^{1,j}_{\floor*{N_1 t}} -  W^{1,j}_{\floor*{N_1 s}} |  &\leq& K_1 \delta + \frac{K_2}{N_1}, \notag \\
|  W^{2,i,j}_{\floor*{N_1 t}} -  W^{2,i,j}_{\floor*{N_1 s}} |  &\leq& K_1 \delta + \frac{K_2}{N_1}. \nonumber
\end{eqnarray}
The desired result then follows.
\end{proof}

We conclude this section now with the required relative compactness of the sequence $\{\gamma^{N_1,N_2}\}_{N_1\in\mathbb{N}}$. This implies that for each fixed $N_2$, every subsequence $\gamma^{N_1,N_2}$'s has a convergent sub-subsequence as $N_1\rightarrow\infty$.

\begin{lemma}\label{L:RelativeCompactness}
The sequence of probability measures $\{\gamma^{N_1,N_2}\}_{N_1\in\mathbb{N}}$ is relatively compact in $D_{E}([0,1])$.
\end{lemma}

\begin{proof}
Given Lemmas \ref{L:CompactContainment} and \ref{L:regularity}, Theorem 8.6 and Remark 8.7 B of Chapter 3 of \cite{EthierAndKurtz}, gives the statement of the lemma.
\end{proof}

\subsection{Identification of the Limit} \label{Identification}

We consider the evolution of the empirical measure $\gamma^{N_1, N_2}_t$ via test functions $f \in C^{2}_{b}(\mathbb{R}^{d+2N_2})$. Using a Taylor expansion based argument similarly to  \cite{NeuralNetworkLLN}, we can show that the scaled empirical measure satisfies
\begin{align}
\la f, \gamma^{N_1, N_2}_t \ra &- \la f, \gamma^{N_1, N_2}_0 \ra = \int_0^t   \int_{\mathcal{X}\times\mathcal{Y}}   \big{(} y -  g_s^{N_1, N_2}(x)   \big{)} \la H_s^{N_1, N_2}(x) \cdot \nabla_c f, \gamma^{N_1,N_2}_{s} \ra \pi(dx,dy)  ds\notag \\
&+ \int_0^t \int_{\mathcal{X}\times\mathcal{Y}}    \big{(} y -  g_s^{N_1,N_2}(x) \big{)} \la  \sigma(w^{1} \cdot x) (\sigma'(Z^{N_1, N_2}_s(x)) \odot c) \cdot \nabla_{w^2} f, \gamma^{N_1, N_2}_{s} \ra \pi(dx, dy) ds \notag \\
&+ \int_0^t \int_{\mathcal{X}\times\mathcal{Y}}    \big{(} y -  g_s^{N_1,N_2}(x) \big{)} \frac{1}{N_2} \sum_{i=1}^{N_2} \la  c_i \sigma'(Z_s^{i, N_1, N_2}(x)) w^{2,i}  \sigma'(w^{1} \cdot x) x \cdot \nabla_{w^{1}} f, \gamma^{N_1, N_2}_{s} \ra \pi(dx, dy) ds \notag \\
&+ M^{N_1, N_2}(t) + O(N_1^{-1}),
\label{EvolutionLayer1}
\end{align}
where $M^{N_1, N_2}(t)$ is a martingale term,
\begin{eqnarray}
Z_s^{i, N_1, N_2}(x) &=& \la w^{2,i} \sigma(w^{1} \cdot x), \gamma_s^{N_1, N_2} \ra, \notag \\
H_s^{i, N_1, N_2}(x) &=&  \sigma( Z_s^{i, N_1, N_2}(x) ), \notag \\
g_s^{N_1, N_2}(x) &=& \frac{1}{N_2}  \sum_{i=1}^{N_2} H_s^{i, N_1, N_2}(x) \la c_i, \gamma_s^{N_1, N_2} \ra.\nonumber
\end{eqnarray}
and we have defined $H_t^{N_1, N_2}$ as the vector $( H_t^{1, N_1, N_2}, \ldots, H_t^{N_2, N_1, N_2})$, $c$ as the vector $(c_1, \ldots, c_{N_2})$, and $Z_t^{N_1, N_2}$ as the vector $(Z_t^{1, N_1, N_2}, \ldots, Z_t^{N_2, N_1, N_2})$. The martingale term $M^{N_1, N_2}(t)$ converges to $0$ in $L^2$ as $N_1 \rightarrow \infty$. That is,
\begin{eqnarray}
\lim_{N_1 \rightarrow \infty} \mathbb{E} \bigg{[} \bigg{(} M^{N_1, N_2}(t) \bigg{)}^2 \bigg{]} = 0.
\label{MartingaleTermL2}
\end{eqnarray}

The proof for (\ref{MartingaleTermL2}) follows as in Lemma 3.1 of  \cite{NeuralNetworkLLN} and thus it is omitted. The limit point of $\gamma^{N_1, N_2}$, as $N_1 \rightarrow \infty$ and for a fixed $N_2$, will satisfy the evolution equation
\begin{align}
\la f, \gamma^{N_2}_t \ra - \la f, \gamma^{N_2}_0 \ra &= \int_0^t   \int_{\mathcal{X}\times\mathcal{Y}}   \big{(} y -  g_s^{N_2}(x)   \big{)} \la H_s^{N_2}(x) \cdot \nabla_c f, \gamma^{N_2}_{s} \ra \pi(dx,dy)  ds\notag \\
&+ \int_0^t \int_{\mathcal{X}\times\mathcal{Y}}    \big{(} y -  g_s^{N_2}(x) \big{)} \la  \sigma(w^{1} \cdot x) (\sigma'(Z_s) \odot c) \cdot \nabla_{w^2} f, \gamma^{N_2}_{s} \ra \pi(dx, dy) ds \notag \\
&+ \int_0^t \int_{\mathcal{X}\times\mathcal{Y}}    \big{(} y -  g_s^{N_2}(x) \big{)} \frac{1}{N_2} \sum_{i=1}^{N_2} \la  c_i \sigma'(Z_s^{i, N_2}(x)) w^{2,i} \sigma'(w^{1} \cdot x) x \cdot \nabla_{w^{1}} f, \gamma^{N_2}_{s} \ra \pi(dx, dy) ds,
\label{LimitSPDELLN}
\end{align}
where
\begin{eqnarray}
Z_s^{i, N_2}(x) &=& \la w^{2,i} \sigma(w^{1}\cdot x), \gamma_s^{N_2} \ra, \notag \\
H_s^{i, N_2}(x) &=&  \sigma( Z_s^{i, N_2}(x) ), \notag \\
g_s^{N_2}(x) &=& \frac{1}{N_2}  \sum_{i=1}^{N_2} H_s^{i, N_2}(x) \la c_i, \gamma_s^{N_2} \ra, \nonumber
\end{eqnarray}
and
\begin{eqnarray}
\gamma^{N_2}_0(dw^{1}, dw^2, dc)  = \mu_{W^{1}}(dw^{1}) \times \mu_{W^2}(dw^{2,1}) \times \cdots \times \mu_{W^2}(dw^{2,N_2})\times \delta_{ C^1_{\circ}}(dc^1) \times \cdots \times   \delta_{ C^{N_2}_{\circ}}(dc^{N_2}). \nonumber
\end{eqnarray}

Let $\pi^{N_1, N_2}$ be the probability measure of $\left(\gamma_t^{N_1, N_2} \right)_{0\leq t\leq 1}$. Each $\pi^{N_1, N_2}$ takes values in the set of probability measures $\mathcal{M} \big{(} D_E([0,1]) \big{)}$. Relative compactness, proven in Section \ref{RelativeCompactness}, implies that there is a subsequence $\pi^{N_{1_k}, N_2}$ which weakly converges. We must prove that any limit point $\pi^{N_2}$ of a convergent subsequence $\pi^{N_{1_k}, N_2}$ will satisfy the evolution equation (\ref{LimitSPDELLN}).
\begin{lemma}
Let $\pi^{N_{1_k}, N_2}$ be a convergent subsequence with a limit point $\pi^{N_2}$. Then, $\pi^{N_2}$ is a Dirac measure concentrated on $\gamma^{N_2} \in D_E([0,1])$ and $\gamma^{N_2}$ satisfies the measure evolution equation (\ref{LimitSPDELLN}).
\end{lemma}
\begin{proof}
We define a map $F(\gamma): D_{E}([0,1]) \rightarrow \mathbb{R}_{+}$ for each $t \in [0,T]$, $f \in C^{2}_{b}(\mathbb{R}^{d+2N_2})$, $g_{1},\cdots,g_{p}\in C_{b}(\mathbb{R}^{d+2N_2})$ and $0\leq s_{1}<\cdots< s_{p}\leq t$.

\begin{eqnarray}
F(\gamma) &=&  \bigg{|} \bigg{(} \la f, \gamma_t \ra - \la f,  \gamma_0 \ra -\int_0^t   \int_{\mathcal{X}\times\mathcal{Y}}   \big{(} y -  g_s^{N_2}(x)   \big{)} \la H_t^{N_2}(x) \cdot \nabla_c f, \gamma^{N_2}_{s} \ra \pi(dx,dy)  ds \notag \\
&-& \int_0^t \int_{\mathcal{X}\times\mathcal{Y}}    \big{(} y -  g_s^{N_2}(x) \big{)} \la  \sigma(w^{1} \cdot x) (\sigma'(Z_{s}) \odot c) \cdot \nabla_{w^2} f, \gamma^{N_2}_{s} \ra \pi(dx, dy) ds \notag \\
&-& \int_0^t \int_{\mathcal{X}\times\mathcal{Y}}   \big{(} y -  g_s^{N_2}(x) \big{)} \frac{1}{N_2} \sum_{i=1}^{N_2} \la  c_i \sigma'(Z_s^{i, N_2}(x)) w^{2,i} \sigma'(w^{1} \cdot x) x \cdot \nabla_{w^{1}} f, \gamma^{N_2}_{s} \ra \pi(dx, dy) ds  \bigg{)} \notag \\
&\times& \la g_{1},\gamma_{s_{1}}\ra\times\cdots\times \la g_{p},\gamma_{s_{p}}\ra\bigg{|} .
\label{FmapMu}
\end{eqnarray}

Then,
\begin{eqnarray}
\mathbb{E}_{\pi^{N_1, N_2}} [ F(\gamma) ] &=& \mathbb{E} [ F( \gamma^{N_1,N_2} ) ] \notag \\
&=& \mathbb{E} \left| \left(M^{N_1, N_2 }(t)  + O(N_1^{-1})\right)\prod_{i=1}^{p} \la g_{i},\gamma^{N_1, N_2}_{s_{i}}\ra \right| \notag \\
&\leq&  \mathbb{E}[ | M^{N_1, N_2}(t) |  ] + O(N^{-1})  \notag \\
&\leq& \mathbb{E}[ ( M^{N_1, N_2}(t) )^2 ]^{1/2} + O(N_1^{-1})  \notag \\
&\leq& K ( \frac{1}{\sqrt{N}}+\frac{1}{N} ).\notag
\end{eqnarray}

Therefore,
\begin{eqnarray}
\lim_{N_1 \rightarrow \infty} \mathbb{E}_{\pi^{N_1, N_2}} [ F(\gamma) ] = 0.\notag
\end{eqnarray}
Since $F(\cdot)$ is continuous and $F( \gamma^{N_1, N_2})$ is uniformly bounded (due to the uniform boundedness results of Section \ref{RelativeCompactness}),
\begin{eqnarray}
 \mathbb{E}_{\pi^{N_2}} [ F(\gamma) ] = 0.\notag
\end{eqnarray}
Since this holds for each $t \in [0,T]$, $f \in C^{2}_{b}(\mathbb{R}^{d+2N_2})$ and $g_{1},\cdots,g_{p}\in C_{b}(\mathbb{R}^{d+2N_2})$, $\gamma^{N_2}$ satisfies the evolution equation (\ref{LimitSPDELLN}).
\end{proof}

It remains to prove that the evolution equation (\ref{LimitSPDELLN}) has a unique solution. This is the content of Section \ref{UniquenessFirstLayer}.

\subsection{Uniqueness} \label{UniquenessFirstLayer}

\begin{lemma} \label{UniquenessFirstLayerLemma}
There exists a unique solution to the evolution equation (\ref{LimitSPDELLN}).
\end{lemma}
\begin{proof}

We only sketch the proof as the argument is similar to the uniqueness argument of \cite{NeuralNetworkLLN}. Consider the particle system:
\begin{eqnarray}
d C_t^i &=&  \int_{\mathcal{X} \times \mathcal{Y} } \big{(} y - g^{N_2}_t(x)  \big{)} H_t^{2,i}(x) \pi(dx,dy) dt, \phantom{.....} C^{i}_{0}= C^{i}_{\circ}, \phantom{....} i = 1, \ldots, N_2, \notag \\
 d W_t^{1} &=&  \int_{\mathcal{X} \times \mathcal{Y} } \big{(} y - g^{N_2}_t(x)  \big{)}  \left( \frac{1}{N_2} \sum_{i=1}^{N_2} C_t^i \sigma'(Z_t^{i}(x) )  W^{2,i}_t  \right) \sigma'(W_t^{1} \cdot x ) x \pi(dx,dy)  dt, \phantom{.....} W_0^{1} \sim \mu_{W^{1}}(dw), \notag \\
d W^{2,i}_{t} &=&   \int_{\mathcal{X} \times \mathcal{Y} } \big{(} y - g^{N_2}_t(x) \big{)} C_t^i \sigma'(Z_t^{i} )  H_t^1(x) \pi(dx,dy)  dt, \phantom{.....} W^{2,i}_0 \sim \mu_{W^2}(dw^2), \phantom{.....} i =1, \ldots, N_2, \notag \\
H_t^{1}(x) &=& \sigma ( W_t^{1} \cdot  x), \notag \\
Z_t^{i}(x) &=& \left< w^{2,i}  H^{1}_t(x), \gamma^{N_2}_{t}\right>, \notag \\
H_t^{2,i}(x) &=& \sigma( Z_t^{i}(x) ), \notag \\
g^{N_2}_t(x) &=& \frac{1}{N_2} \sum_{i=1}^{N_2} C_t^i H_t^{2,i}(x).
\label{SystemLimitLayer1}
\end{eqnarray}

Let $\nu_{t,(c_1, \ldots, c_{N_2})}$ be the conditional law of $(W_t^{1}, W^{2,1}_t, \ldots W^{2,N_2}_t, C_t^{1}, \ldots, C_t^{N_2} )_{0 \leq t \leq T}$ given $(C^1_{\circ}, \ldots, C^{N_2}_{\circ}) = (c_1, \ldots, c_{N_2})$. Similarly to the general results of \cite{Kolokoltsov},  $\nu_{t,(C^1_{\circ}, \ldots, C^{N_2}_{\circ})}=\gamma^{N_2}_{t}$ is a solution to the evolution equation (\ref{LimitSPDELLN}) and conversely, any solution $\nu_{t,(C^1_{\circ}, \ldots, C^{N_2}_{\circ})}$ to the evolution equation (\ref{LimitSPDELLN}) must also be the law of the solution to (\ref{SystemLimitLayer1}).

Let us next prove that we can write $Z_t^{i}(x)=\mathbb{E} \bigg{[} W^{2,i}_t  H^{1}_t(x) \bigg{|} C^1_{\circ}, \ldots, C^{N_2}_{\circ} \bigg{]}$. Recall that $\mathcal{F}_C^{N_2} = ( C^1_{\circ}, \ldots, C^{N_2}_{\circ})$ and let us similarly define $\mathcal{F}_{W^{2}}^{N_2} = ( W^{2,1}_{0},  \ldots, W^{2,N_2}_{0})$. Inspecting (\ref{SystemLimitLayer1}) it becomes clear that the random variables $(W_t^{1}, W^{2,1}_t, \ldots W^{2,N_2}_t, C_t^{1}, \ldots, C_t^{N_2} )_{0 \leq t \leq T}$ depend, in addition to their own initial conditions, also on $\mathcal{F}_C^{N_2}$ and $\mathcal{F}_{W^2}^{N_2}$ in a symmetric way through the terms $g^{N_2}_t(x)$ and $\frac{1}{N_2} \sum_{i=1}^{N_2} C_t^i \sigma'(Z_t^{i}(x) )  W^{2,i}_t$. In order to make this dependence specific in the notation, we write in particular that
\[
W^{1}_{t}=W^{1,W_{0};\mathcal{F}_C^{N_2},\mathcal{F}_{W^2}^{N_2}}_{t} \text{, }W^{2,i}_{t}=W^{2,C^i_{\circ}, W_{0},W^{2,i}_{0};\mathcal{F}_C^{N_2},\mathcal{F}_{W^2}^{N_2}}_{t} \text{ and } C^{i}_{t}=C^{C^i_{\circ};\mathcal{F}_C^{N_2}}_{t}.
\]

This, then leads to the following calculations
\begin{align}
Z_t^{i}(x) &= \left< w^{2,i}  H^{1}_t(x), \gamma^{N_2}_{t}\right>\nonumber\\
&=\mathbb{E}\bigg{[} W^{2,i}_t  \sigma ( W_t^{1} \cdot  x) \bigg{]}\nonumber\\
&=\mathbb{E}\bigg{[}W^{2,C^i_{\circ}, W_{0},W^{2,i}_{0};\mathcal{F}_C^{N_2},\mathcal{F}_{W^2}^{N_2}}_{t} \sigma ( W^{1,W_{0};\mathcal{F}_C^{N_2},\mathcal{F}_{W^2}^{N_2}}_{t} \cdot  x) \bigg{]}\nonumber\\
&=\mathbb{E}\bigg{[} \mathbb{E}\bigg{[}W^{2,C^i_{\circ}, W_{0},W^{2,i}_{0};\mathcal{F}_C^{N_2},\mathcal{F}_{W^2}^{N_2}}_{t}  \sigma ( W^{1,W_{0};\mathcal{F}_C^{N_2},\mathcal{F}_{W^2}^{N_2}}_{t} \cdot  x) \bigg{|} W_{0}, \mathcal{F}_{W^{2}}^{N_2},\mathcal{F}_C^{N_2} \bigg{]}\bigg{]}\nonumber\\
&=\left<W^{2,C^i_{\circ}, W_{0},W^{2,i}_{0};\mathcal{F}_C^{N_2},\mathcal{F}_{W^2}^{N_2}}_{t} \sigma (W^{1,W_{0};\mathcal{F}_C^{N_2},\mathcal{F}_{W^2}^{N_2}}_{t} \cdot  x), \gamma_{0}^{N_2}\right>\nonumber\\
&= \mathbb{E}\bigg{[}W^{2,C^i_{\circ}, W_{0},W^{2,i}_{0};\mathcal{F}_C^{N_2},\mathcal{F}_{W^2}^{N_2}}_{t}  \sigma ( W^{1,W_{0};\mathcal{F}_C^{N_2},\mathcal{F}_{W^2}^{N_2}}_{t} \cdot  x) \bigg{|} \mathcal{F}_C^{N_2} \bigg{]}\nonumber\\
&=\mathbb{E} \bigg{[} W^{2,i}_t  \sigma ( W_t^{1} \cdot  x) \bigg{|} C^1_{\circ}, \ldots, C^{N_2}_{\circ} \bigg{]}\nonumber
\end{align}

The second to the last equality is due to the assumed independence of the initial conditions via Assumption \ref{A:Assumption1} together with the fact that the marginal of $\gamma_{0}^{N_2}$ with respect to $C^1_{\circ}, \ldots, C^{N_2}_{\circ}$ is a product of delta Dirac distributions.

It remains to show that the solution to (\ref{SystemLimitLayer1}) is unique. We do this via a fixed point argument, completely analogous to  Section 4 of  \cite{NeuralNetworkLLN}. The details are omitted due to the similarity of the argument.

\end{proof}

\subsection{Convergence Result for First Layer}  \label{ConvergenceFirstLayer}

Let $\pi^{N_1, N_2}$ be the probability measure corresponding to $\gamma^{N_1, N_2}$. Each $\pi^{N_1, N_2}$ takes values in the set of probability measures $\mathcal{M} \big{(} D_E([0,1]) \big{)}$. Relative compactness, proven in Section \ref{RelativeCompactness}, implies that every subsequence $\pi^{N_{1_k}, N_2}$ has a further sub-sequence $\pi^{N_{1_{k_m}}, N_2}$ which weakly converges. Section \ref{Identification} proves that any limit point $\pi^{N_2}$ of $\pi^{N_{1_{k_m}}, N_2}$ will satisfy the evolution equation (\ref{LimitSPDELLN}). Section \ref{UniquenessFirstLayer} proves that the solution of the evolution equation (\ref{LimitSPDELLN}) is unique. Therefore, by Prokhorov's Theorem, $\pi^{N_1, N_2}$ weakly converges to $\pi^{N_2}$, where $\pi^{N_2}$ is the distribution of $\gamma^{N_2}$, the unique solution of (\ref{LimitSPDELLN}). That is, $\gamma^{N_1, N_2}$ converges in distribution to $\gamma^{N_2}$.

\bibliographystyle{plainnat}


\renewcommand{\refname}{Bibliography}

\end{document}